\documentclass[11pt]{amsart}

\usepackage{amsmath,amssymb}
\usepackage{amsthm}
\usepackage[nobysame]{amsrefs}
\usepackage{latexsym}
\usepackage{indentfirst}
\usepackage{mathrsfs}
\usepackage{graphicx,subfigure}
\usepackage[margin=1in]{geometry}
\usepackage{xcolor}
\usepackage{hyperref}

\numberwithin{equation}{section}

\theoremstyle{plain}
\newtheorem{theorem}{Theorem}[section]
\newtheorem{lemma}[theorem]{Lemma}

\newtheorem{prop}[theorem]{Proposition}

\theoremstyle{definition}

\newtheorem{assump}{Assumption}

\theoremstyle{remark}
\newtheorem*{remark}{Remark}
\newtheorem{example}[theorem]{Example}

\DeclareMathOperator{\dist}{dist}

\DeclareMathOperator{\spec}{spec}
\DeclareMathOperator{\supp}{supp}

\newcommand{\secref}[1]{\S\ref{#1}}
\newcommand{\figref}[1]{Figure~\ref{#1}}
\newcommand{\assumpref}[1]{Assumption~\ref{#1}}
\newcommand{\exref}[1]{\textbf{Example~\ref{#1}}}

\newcommand{\ud}{\,\mathrm{d}}
\newcommand{\RR}{\mathbb{R}}

\newcommand{\CC}{\mathscr{C}}
\newcommand{\Or}{\mathcal{O}}

\newcommand{\wt}[1]{\widetilde{#1}}

\newcommand{\mc}[1]{\mathcal{#1}}

\newcommand{\eps}{\epsilon}
\newcommand{\Om}{\Omega}
\newcommand{\Lam}{\Lambda}
\newcommand{\lam}{\lambda}

\newcommand{\dx}{\mathrm{d}x}
\newcommand{\na}{\nabla}

\newcommand{\al}{\alpha}

\newcommand{\lr}[1]{\bigl(#1\bigr)}
\newcommand{\abs}[1]{\lvert#1\rvert}
\newcommand{\norm}[1]{\lVert#1\rVert}

\newcommand{\Veff}{V_{\mathrm{eff}}}

\newcommand{\barint}{\kern4pt \raise3.4pt\hbox{\vrule height.6pt
    width7pt} \kern-11pt \int}

\begin{document}

\title{Analysis of the divide-and-conquer method for electronic
  structure calculations}

\author{Jingrun Chen}
\address{Mathematics Department, South Hall 6705, University of California, Santa Barbara, CA 93106}
\email{cjr@math.ucsb.edu}

\author{Jianfeng Lu}
\address{Departments of Mathematics, Physics, and Chemistry, Duke University, Box 90320, Durham, NC 27708}
\email{jianfeng@math.duke.edu}

\begin{abstract}
  We study the accuracy of the divide-and-conquer method for
  electronic structure calculations.  The analysis is conducted for a
  prototypical subdomain problem in the method. We prove that the
  pointwise difference between electron densities of the global system
  and the subsystem decays exponentially as a function of the distance
  away from the boundary of the subsystem, under the gap assumption of
  both the global system and the subsystem. We show that gap
  assumption is crucial for the accuracy of the divide-and-conquer
  method by numerical examples. In particular, we show examples with
  the loss of accuracy when the gap assumption of the subsystem is
  invalid.
\end{abstract}

\date{\today}
\maketitle

\section{Introduction}

Many systems in materials science, chemistry and other areas are
greatly influenced by the electronic structure, which requires the
full quantum-mechanical description. However, directly solving the
quantum many-body problem for real systems is impractical even with
the present supercomputers since a $3N$-dimensional
antisymmetric wave function is needed to describe a system with $N$
electrons. Lots of electronic structure models, which aim at
approximating the solution of many-body Schr\"odinger equations, have
been proposed.

Kohn-Sham density functional theory (DFT) \cites{Hk, Ks, ParrYang:89,
  M04} is one of the most popular and successful tools for electronic
structure analysis, in which $N$ one-particle wave functions are used
to describe the N-electron system with properly approximated energy
functionals. The corresponding Kohn-Sham equations are a system of
nonlinear eigenvalue problems. To solve the nonlinear eigenvalue
equation, the self-consistent field iteration is often used.  The
electron density is updated at each iteration until self-consistency
is achieved. The computational cost of each iteration step for
conventional algorithm scales as $\Or(N^3)$ due to diagonalization and
orthogonalization. For different systems, the required number of
iterations might scale differently and depends on the choice of mixing
techniques. The total cost of solving the Kohn-Sham equations scales
at least as $\Or(N^3)$. Such computational scaling is prohibitively
expensive when the number of electrons is large.

Many efforts have been devoted to design linear scaling methods, i.e.,
$\Or(N)$ methods, for electronic calculations within the framework of
Kohn-Sham DFT over the past twenty years (see e.g.~\cites{G99,
  Bowler}). These methods share the common ground of exploiting the
locality, or nearsightedness \cites{K96, ProdanKohn} to reduce the
computational complexity.  Locality here means the dependence of the
electron density on the environment decays in distance.  The first
linear scaling method is the divide-and-conquer (DAC) method
proposed by Weitao Yang \cites{Y91a, Y91b}, where the global system is
divided into several subsystems, and each subsystem is solved
separately with atomic orbitals. The electron density of the global
system is then found by getting a global equilibrium condition for the
Fermi energy. In each self-consistent iteration, the cost of DAC
method depends on the number of subsystems which is proportional to
the number of electrons. The DAC method scales as $\Or(N)$ naturally
if the self-consistent field iteration is independent of the
considered system.

In this article, we aim at understanding the accuracy of the
 DAC method, as one of the popular approaches of linear
scaling algorithms. We note that the main idea of the algorithm is
quite similar to the domain decomposition type method, commonly used
in numerical solutions to PDEs. The goal is to understand the accuracy
of the method and the conditions under which the method works.
A key component of the
analysis is to understand the locality of electronic structure from a
mathematical point of view. The main ingredients are geometric
resolvent identity and a Combes-Thomas type decay estimate of the
Green's function.

In the DAC method of electronic structure calculations,
the subsystem can be understood as the global system under certain
(not necessarily small) perturbations.  It turns out that the accuracy
of the method depends crucially on the gap structure of the system and
of the subsystem.  We examine the gap assumption in cases when it is
valid and invalid carefully with numerous examples.  Let us also point
out that our analysis does not assume any particular way of
restriction of the Hamiltonian onto a sub-domain (besides that the gap
assumption is satisfied). This flexibility allows the analysis to be
generalized to a variety of methods in electronic structure
calculations based on the domain decomposition idea.

The outline of the paper is as follows. The detailed description of
the DAC method is presented in \secref{dd}. The
accuracy of the method is analyzed in \secref{accuracy}. By examples
in one dimension and two dimensions, we demonstrate the accuracy of the
method when the gap assumption is valid, and the loss of the accuracy when
the gap assumption is invalid in \secref{gap}.

\section{Divide-and-conquer method}\label{dd}

\subsection{Kohn-Sham density functional theory}
Consider a system of $N_c$ nuclei and $N$ electrons. A set of
one-particle wave functions $\{\psi_k(x)\}_{k=1}^N$ is employed to
represent the interacting electrons in Kohn-Sham DFT. At zero
temperature, the Kohn-Sham energy functional can be written as (for
simplicity of the presentation, we will ignore the spin degeneracy
here and in sequel)
\begin{equation}\label{KSene}
\begin{aligned}
  E_{\mathrm{KS}}[\{\psi_k\}_{k=1}^N]&
  = \sum_{k=1}^{N}\int_{\RR^3}\dfrac12\abs{\na\psi_k}^2\dx
  + \int_{\RR^3}V(x)\rho(x)\dx\\
  &\quad+\dfrac12\iint_{\RR^3 \times \RR^3}
  \dfrac{\lr{\rho-m}(x)\lr{\rho-m}(x')}{\abs{x-x'}}\dx\dx'
  +E_{\mathrm{XC}}[\rho],
\end {aligned}
\end{equation}
where the electron density is given by
\begin {equation}\label{den}
  \rho(x)= \sum_{k=1}^N\abs{\psi_k(x)}^2,
\end {equation}
and the ionic function takes the form
\begin {equation}\label{ionic}
  m(x)=\sum_{k=1}^{N_c} m^a(x-R_k),
\end {equation}
where $m^a$ is a localized smooth function and $\{R_k\}_{k=1}^{N_c}$
are the positions of nuclei, i.e., we have taken a local
pseudopotential for the electron-nucleus interaction \cite{M04} for
simplicity. Our results can be generalized to nonlocal
pseudopotential, but we will not go into the details.

The Kohn-Sham energy functional is minimized with the orthonormal
constraints of the orbitals
\begin {equation}\label{ortho}
  \int_{\RR^3}\psi_k(x)\psi_l(x)\dx=\delta_{kl}, \qquad k,l=1,2,\ldots, N.
\end {equation}
The Euler-Lagrange equation, known as the Kohn-Sham equation, can be
written as
\begin {equation}\label{KSequation}
  H(\rho) \psi_k(x)=\eps_k\psi_k(x), \qquad k=1,2,\ldots, N,
\end {equation}
where $H(\rho) =-\dfrac12\Delta+\Veff$ with
$\Veff=V(x)+\int_{\RR^3}\frac{\rho(x')-m(x')}{\abs{x-x'}}\dx'
+V_{\mathrm{XC}}[\rho]$ and
$V_{\mathrm{XC}}[\rho] = \frac{\delta E_{\mathrm{XC}}[\rho]}{\delta\rho}$.
Here, $\eps_k$ are a set of eigenvalues,
increasingly ordered, and $\{\psi_k\}$ are the associated
eigenfunctions of the effective Hamiltonian. Note that this is a
nonlinear eigenvalue problem as the effective Hamiltonian $H$ depends
on the density, which in turn, depends on the eigenfunctions.

To solve the Kohn-Sham equation \eqref{KSequation}, a self-consistent
iteration is usually employed. At each iterate, for the current guess
of the density $\rho$, we solve for the eigenvalue problem of
$H(\rho)$ to find the first $N$ eigenpairs $\{\eps_k, \psi_k\}$. From
the eigenfunctions, we form a new density $\rho_{\text{new}} = \sum_k
\abs{\psi_k}^2$. The nonlinear iteration is used to find a fixed point
of the map from $\rho$ to $\rho_{\text{new}}$, which is known as the
Kohn-Sham map (see e.g., \cite{ELu:2013}).

The algorithmic bottleneck of the above procedure is to evaluate the
density $\rho_{\text{new}}$ given a Hamiltonian: For a fixed
Hamiltonian $H=-\frac{1}{2}\Delta+V(x)$ with some effective potential
$V \in L^{\infty}$ (consequently, we will take $\mc{D}(H) =
  H^2(\RR^3)$), we look for the square sum of its first $N$
eigenfunctions,
\begin {equation}\label{KSeqn}
  H \psi_k(x)=\eps_k\psi_k(x),\quad k=1,\ldots, N,
\end {equation}
which is a linear eigenvalue problem.  A conventional diagonalization
of the discretized Hamiltonian to solve \eqref{KSeqn} leads to
computational cost that scales cubicly with respect to the number of
electrons. However, the eigenfunctions $\{\psi_k\}$ are just an
intermediate step for the electron density $\rho = \sum_k
\abs{\psi_k}^2$. It is therefore possible to design efficient
algorithms that avoid the eigenvalue problem on the whole
computational domain. One such strategy is the DAC method, which aims
to achieve linear scaling cost for computing the density.

\subsection{Divide-and-conquer method}

The idea of using the DAC method to study electron
structures was firstly proposed by Weitao Yang in \cites{Y91a, Y91b},
which was based on a localized Hamiltonian formulation.  It was then
generalized to a density-matrix formulation \cite{Yl95}. Some recent
developments of the DAC method, or more generally,
domain decomposition type method, can be found in \cites{Wzm, Zmw,
  Barraultetal:2007, Bencteuxetal:2008}.  A great advantage of the
method lies on the intrinsic parallel properties between subdomains,
which has been investigated for large scale calculations with more
than $10^6$ atoms and $10^{12}$ electronic degrees of freedom
\cites{KobayashiNakai:09, OhbaOgata:12, Shimojo:2008, Shimojo:2011}.
In what follows, we describe the main idea of the DAC
method, in the spirit of \cite{Y91a}. To clearly present the method,
we will stay on the PDE level and formulate the algorithm in terms of
operators, rather than first imposing a discretization of the
Hamiltonian. This way, we can separate the error caused by the
 DAC method and by a numerical discretization of the continuous
problem.

The DAC method for electronic structure calculations
involves the following steps. Let us denote the whole computational
domain as $\Omega$. Our goal is to find its corresponding density of
the Hamiltonian $H$ on the whole domain.
\begin{itemize}
\item[Step 1.] Define a partition of domain, $\{\Lam_{\al}\}$, and a
  partition of unity subordinate to the open covering $\{\Lam_{\al},
  p_{\al}\}$.  Usually neighboring subdomains intersect, i.e.,
  $\Lam_{\al}\cap\Lam_{\al'}\neq \emptyset$ when $\al\neq\al'$.
  Nonnegative partition functions
  $\{p_{\al}\}$ satisfy $\sum_{\al}p_{\al}(x)=1, \forall x\in\Om$.


\item[Step 2.] Restrict the Hamiltonian on the domain
  $\Lambda_{\alpha}$ with certain boundary conditions and solve the
  eigenvalue problem in each subsystem
  \begin{align}\label{dc}
    H_{\Lam_{\al}}\psi_k^{\al}(x)=\eps_k^{\al}\psi_k^{\al}(x),\quad x\in\Lam_{\al},
  \end{align}
  where $H_{\Lambda_{\al}}$ denotes the restriction of the
    Hamiltonian, whose domain is a subset of the Sobolev space
    $H^2(\Lambda_{\alpha})$ with prescribed boundary conditions.
\item[Step 3.] Determine the Fermi energy $\eps_F$ by solving the
  equation of charge equilibrium
  \begin {equation}\label{constrain}
    N=\sum_{\al}\sum_k f_{\beta}(\eps_F-\eps_k^{\al})\int_{\Lam_{\al}}
    p_{\al}(x)\abs{\psi_k^{\al}(x)}^2 \dx,
  \end {equation}
  where $f_{\beta}(\eps)=(1+e^{\beta(\eps-\eps_F)})^{-1}$ is the
  Fermi-Dirac function and $\beta=\dfrac1{k_BT}$ with $k_B$
  Boltzmann constant and $T$ absolute temperature.

\item[Step 4.] Construct the electron density
\begin {equation}\label{density}
  \rho^{\mathrm{DAC}}(x)=\sum_{\al}p_{\al}(x)\rho_{\alpha}(x),
\end {equation}
where $\rho_{\alpha}(x)=\sum_k  f_{\beta}(\eps_F-\eps_k^{\al})\abs{\psi_k^{\al}(x)}^2$,
and total energy
\begin{equation}\label{energy}
\begin{aligned}
  E&=\sum_{\al}\sum_k\eps_k^{\al}
  f_{\beta}(\eps_F-\eps_k^{\al})\int_{\Lam_{\al}}p_{\al}(x)\abs{\psi_k^{\al}(x)}^2\dx
  -\frac12\int_{\Om}\int_{\Om}\frac{\rho^{\mathrm{DAC}}(x)\rho^{\mathrm{DAC}}(x')}{\abs{x-x'}}\dx\dx'\\
  &\quad+\frac12\int_{\Om}\int_{\Om}\frac{m(x)m(x')}{\abs{x-x'}}\dx\dx'
  -\int_{\Om}V_{\mathrm{XC}}[\rho^{\mathrm{DAC}}]\rho^{\mathrm{DAC}}(x)\dx+E_{\mathrm{XC}}[\rho^{\mathrm{DAC}}].
\end {aligned}
\end{equation}

\end {itemize}
Note that the above formulation corresponds to a finite temperature
calculation, as considered in the original DAC method
\cite{Y91a}. In practice, if interested in the zero temperature
calculation, we may choose $\beta$ so large that the Fermi-Dirac
function becomes approximately a Heaviside function. In the following
analysis and numerical examples, we will consider the zero temperature
case to focus on the key idea. Our analysis can be extended to finite
temperature situation.

\subsection{A prototypical subsystem problem}
From an analytical point of view, we can just focus on one subsystem
problem from the divide-and-conquer method. The analysis for other
subdomains proceeds in the same fashion and the error of the method
over the whole domain can be controlled by those of the sub-domains
using triangle inequality and observing that $p_{\alpha}$ is a
partition of unity. More specifically, note that the global density in
the DAC method is obtained by $\rho^{\mathrm{DAC}}(x) = \sum_{\alpha}
p_{\alpha}(x) \rho_{\alpha}(x)$ where $\rho_{\alpha}$ is the electron
density calculated in $\Lambda_{\alpha}$.  Denoting $\rho$ the true
density, we have
\begin{equation*}
  \begin{aligned}
    \norm{\rho - \rho^{\mathrm{DAC}}}_{L^{\infty}} & = \sup_x \Bigl\lvert \sum_{\alpha} p_{\alpha}(x) \bigl(\rho(x) - \rho_{\alpha}(x)\bigr)\Bigr\rvert \\
    & \leq \sup_{\alpha}\, \norm{\rho - \rho_{\alpha}}_{L^{\infty}(\Lambda_{\alpha})}.
  \end{aligned}
\end{equation*}
Error estimate in other norms can be similarly obtained.

Let us reformulate the DAC idea for a single domain.
Let $\Lambda$ be a subdomain and let $\Lambda_b$ be a buffer
region surrounding $\Lambda$. In terms of the algorithm in the
previous section, $\Lambda_b$ corresponds to one of the
$\{\Lambda_{\alpha}\}$, and $\Lambda$ is the support of $p_{\alpha}$,
which we choose to be strictly inside. Later in the analysis, we will
also need a slightly smaller buffer region $\wt{\Lambda}_b$ inside of
$\Lambda_b$.  These sets satisfy
$\Lambda\subset\wt{\Lambda}_b\subset\Lambda_b\subset\Om$ with some
distance separating their boundaries, see \figref{fig:domain1} for
a schematic picture.

For a prescribed Fermi energy $\eps_F$, we are interested in the
density over the domain $\Lambda$, calculated by solving the
eigenproblem on $\Lambda_b$. Namely, we define
\begin{equation}\label{denbuffer}
  \rho_{\Lambda}(x) = \sum_{\eps_k \leq \eps_F}|\psi_k(x)|^2, \quad x\in\Lambda,
\end{equation}
where the eigenpairs $(\eps_k,\psi_k)$ are obtained by solving the
following eigenvalue problem in $\Lambda_b$
\begin{equation}\label{eigbuffer}
  H_{\Lambda_b}\psi_k = \eps_k\psi_k, \quad k=1,2,\ldots.
\end{equation}
To understand the accuracy of the DAC method, it then suffices
to understand the difference between $\rho_{\Lambda}$
and the exact density $\rho$ restricted on $\Lambda$.

\begin{figure}[htbp]
\vspace{-6em}
\centering
\subfigcapskip -0.9in
\subfigure[]{%
\label {fig:domain1}
\includegraphics[width=2.5in]{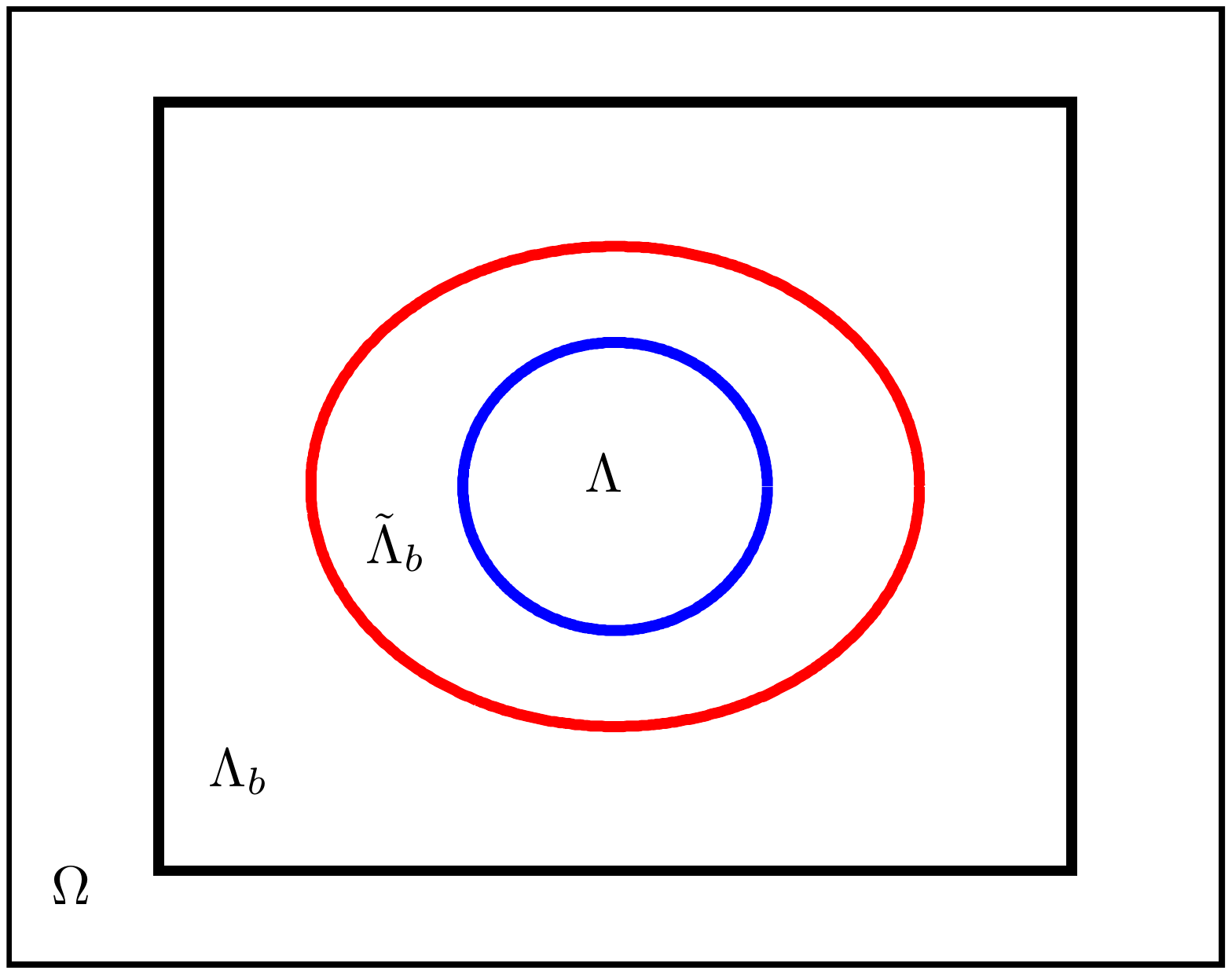}}%
\subfigure[]{
\label {fig:domain2}
\includegraphics[width=2.5in]{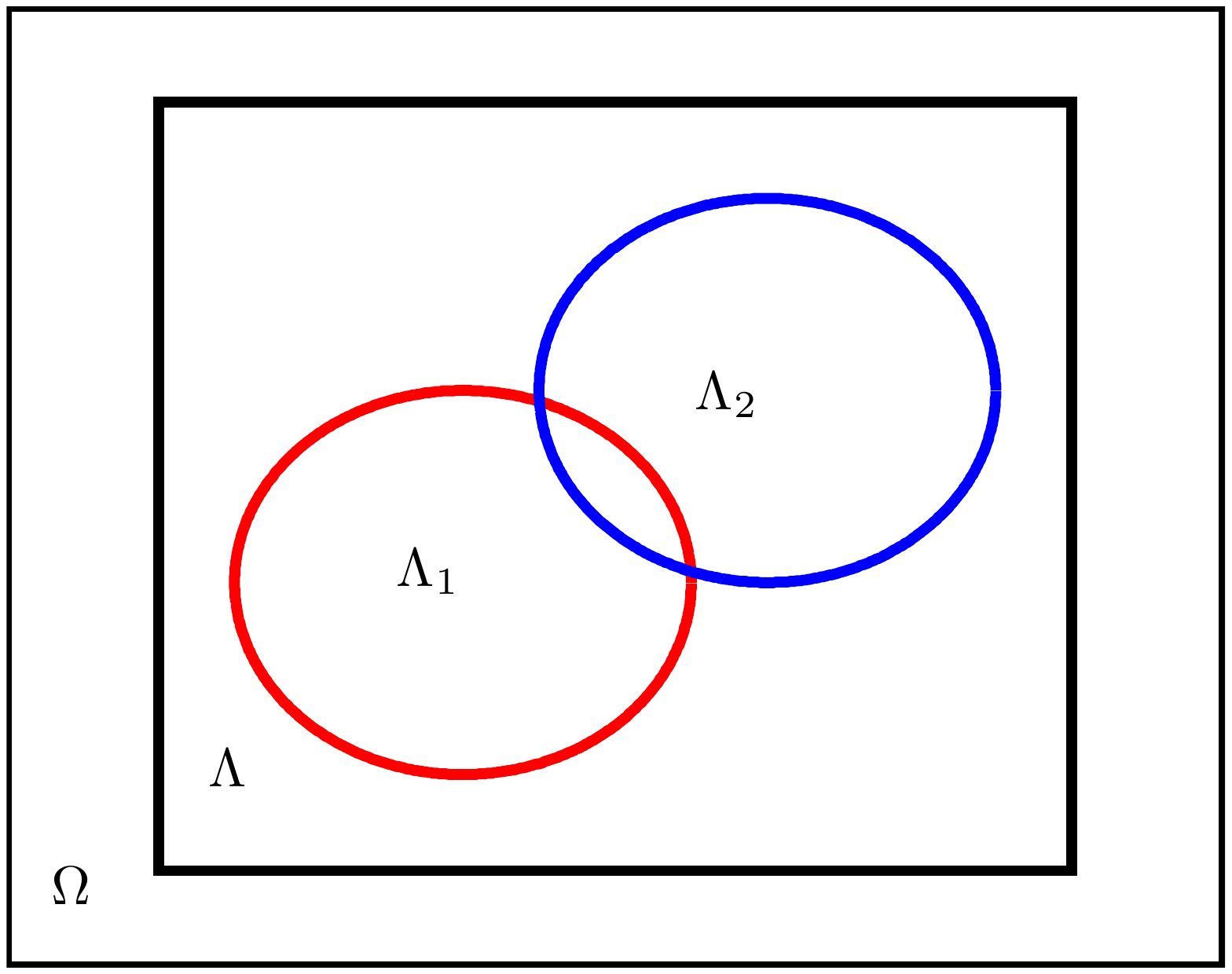}}%
\vspace{-4em}
\caption{\small Schematic pictures of domain and subdomains in the
  DAC method and the geometric resolvent
  identity. (a) A prototypical subdomain problem in the DAC
  method; (b) Domain and subdomains in the geometric
  resolvent identity Lemma~\ref{lem:GRI}.}\label {fig:domain}
\end{figure}

\section{Accuracy of the method}\label{accuracy}

The main tool we will use is the geometric resolvent identity and the
decay estimate of the Green's functions.  The geometric resolvent
identity relates the Green's function defined on a subdomain to the
Green's function on a larger domain. For a domain $\Lambda$, we will
denote $\Lambda^c$ its complement; and for two sets $A$ and $B$,
$\dist(A, B) = \inf_{x \in A, y \in B} \dist(x, y)$.

\begin{lemma}[Geometric Resolvent Identity] \label{lem:GRI} Consider
  four open sets $\Lam_1, \Lam_2, \Lam$ and $\Om$ that satisfy
  $\Lam_1\subset\Lam, \Lam_2\subset\Lam, \Lam\subset\Om$ and
  $\dist\{\Lam_1\cup\Lam_2, \Lam^c\}>0$ (see \figref{fig:domain2} for
  an illustration of these sets). Let $\Theta$ be a smooth function
  which is identically $1$ on a neighborhood of $\Lam_1\cup\Lam_2$ and
  identically $0$ on a neighborhood of $\Lam^c$. Given any restriction $H_{\Om}$ and $H_{\Lam}$ of $H$ to
  $\Om$ and $\Lam$, respectively, we have
\begin {equation}\label{gri1}
1_{\Lam_1}(H_{\Om}-\lam)^{-1}=1_{\Lam_1}(H_{\Lam}-\lam)^{-1}\Theta
+ 1_{\Lam_1}(H_{\Lam}-\lam)^{-1}[H, \Theta](H_{\Om}-\lam)^{-1}
\end {equation}
for any $\lam$ for which both resolvents exist. Also
\begin {equation}\label{gri2}
1_{\Lam_1}(H_{\Om}-\lam)^{-1}1_{\Lam_2}
=1_{\Lam_1}(H_{\Lam}-\lam)^{-1}1_{\Lam_2}
+ 1_{\Lam_1}(H_{\Lam}-\lam)^{-1}[H, \Theta](H_{\Om}-\lam)^{-1}1_{\Lam_2},
\end {equation}
under the same conditions.
\end{lemma}

\begin{proof}
  The lemma is well-known in the analysis of Schr\"odinger operators
  and its proof is standard (see e.g., \cite{Aizenman}*{Lemma 4.2}). We
  include the short proof here for completeness.
  First note the identity
  \begin{equation}
    [H, \Theta] = (H_{\Lambda} - \lambda) \Theta - \Theta (H_{\Omega} - \lambda)
  \end{equation}
  since $\supp \Theta \subset \Lambda \subset \Omega$. The identity
  \eqref{gri1} follows from multiplying on the left by $1_{\Lambda_1}
  (H_{\Omega} - \lambda)^{-1}$ and on the right by $(H_{\Omega} -
  \lambda)^{-1}$.  The identity \eqref{gri2} follows from \eqref{gri1}
  by applying $1_{\Lambda_2}$ on the right on both hand sides.
\end{proof}

Let us recall the spectral representation of the electron density (see
e.g., \cite{ELu:2013})
\begin{equation}
  \rho(x) = \frac{1}{2\pi i} \int_{\CC} (\lambda - H)^{-1} \ud \lambda(x, x),
\end{equation}
where the right hand side stands for the diagonal of the kernel of the
operator $(2 \pi i)^{-1} \int_{\CC} (\lambda - H)^{-1} \ud
\lambda$. Here $\CC$ is a contour in the complex plane that separates
the occupied spectrum of $H$ (the eigenvalues below the Fermi energy
$\eps_F$ with the rest of the spectrum). In the DAC
method, this is approximated by
\begin{equation}
  \rho_{\Lambda}(x) = \frac{1}{2\pi i} \int_{\CC} (\lambda - H_{\Lambda_b})^{-1} \ud \lambda (x, x),
\end{equation}
where $\Lambda_b$ is a buffer region surrounding $\Lambda$. Without
loss of generality, we will assume that the buffer satisfies
$\dist(\Lambda, \Lambda_b^c) \geq 2$. We will also
define the region
\begin{equation}\label{eq:wtLambdab}
  \wt{\Lambda}_b  = \bigl \{ x \in \Lambda_b \mid \dist(x, \Lambda_b^c) \leq 1 \bigr \}.
\end{equation}
By construction, it is clear that we have $\dist(\Lambda,
\wt{\Lambda}_b^c) \geq 1$. We note that the distances $1$ and $2$ are
chosen here merely for convenience, any finite $\Or(1)$ distance will
work, though the final constants in the estimate depend on how
separated the domains are.

We may proceed to compare the pointwise values of $\rho$ and
$\rho_{\Lambda}$ by using results on regularity estimate of Green's
function for elliptic operators (e.g.,~\cite{Agmon:1965} and
\cite{ELu:2013}*{Lemma 6.4}).  Here, for
simplicity of presentation and to better convey the key idea, we will
instead work with the following locally mollified version of the
densities (with slight abuse of notations, we still denote them as
$\rho$ and $\rho_{\Lambda}$)
\begin{align}
  & \rho(x) = \frac{1}{2\pi i} \int_{\CC} \bigl\langle \varphi_x, (\lambda - H)^{-1} \varphi_x \bigr\rangle \ud \lambda, \\
  & \rho_{\Lambda}(x) = \frac{1}{2\pi i} \int_{\CC} \bigl\langle \varphi_x,  (\lambda - H_{\Lambda_b})^{-1} \varphi_x \bigr\rangle \ud \lambda,
\end{align}
where $\varphi_x$ is a fixed numerical delta function centered at
$x$. For simplicity of notation, we will also abuse the notation by
writing $\dist(x, A) := \dist(\supp \varphi_x, A)$ for a set $A$.
Note that the mollification is in agreement with practical numerical
implementations, since some discretization will be used for the
Hamiltonian operator. Other forms of $\varphi_x$, such as averaging in
a small ball around $x$, can also be used. Accuracy of the method is
the same with a possibly different constant.

In general, the restriction of $H$ onto the domain $\Lambda_b$ might
dramatically change the spectrum of the operator. As will be shown in
the numerical examples, without any assumption on the spectral
properties of the truncated operator $H_{\Lambda_b}$, the accuracy of
the method is not guaranteed, in particular, the difference between
$\rho(x)$ and $\rho_{\Lambda}(x)$ might be quite large and decay very
slowly when $x$ is moving inside $\Lambda$ away from the boundary
$\partial \Lambda$. To guarantee the fast decay of the error, we make
the following gap assumption for the truncated system $H_{\Lambda_b}$.

\begin{assump}[Gap assumption]\label{assump}
  Let $\spec_{\text{occ}}(H)$ and $\spec_{\text{unocc}}(H)$ be the
  occupied and unoccupied spectra of $H$ respectively. We assume that
  there exists $\eps_F$ and $e_g > 0$ such that
  \begin{align}
      \label{assumpa}
    & \eps_F - e_g / 2 \geq \sup \spec_{\text{occ}}(H); \\
      \label{assumpb}
    & \eps_F + e_g / 2 \leq \inf \spec_{\text{unocc}}(H); \\
      \label{assumpc}
    & (\eps_F - e_g / 2, \eps_F + e_g / 2) \cap \spec(H_{\Lambda_b}) =
    \emptyset.
  \end{align}
\end{assump}
Note that, $e_g$ might be smaller than the spectral gap between
occupied and unoccupied spectra of $H$. Physically, the assumption
means that the restriction of the Hamiltonian operator on the
subsystem preserves the gap around the Fermi energy.  In particular,
the assumption implies the existence of a contour $\CC$ such that
\begin{equation*}
  \dist(\CC, \spec(H)) \geq e_g/2 \qquad \text{and} \qquad
  \dist(\CC, \spec(H_{\Lambda_b})) \geq e_g/2.
\end{equation*}

\begin{remark}
  If Assumption~\ref{assump} is satisfied by all the sub-domains, we
  can then find a uniform gap in the spectra of all sub-domain
  Hamiltonians. This means that the Fermi level can be chosen
  uniformly for all the sub-domain, which gives the choice of the
  global Fermi energy in Step 3 of the DAC algorithm.
\end{remark}

\begin{theorem}[Accuracy of the method] \label{thm:accuracy} Under
  Assumption~\ref{assump}, there exist constants $C$ and $\gamma$ such
  that
  \begin{equation}\label{expdecay}
    \abs{\rho(x) - \rho_{\Lambda}(x)} \leq C e^{-2 \gamma (\dist(x, \Lambda_b^c) - 1)}, \qquad \forall \, x \in \Lambda.
  \end{equation}
  The constants $C$ and $\gamma$ depend only on $\eps_F$, $e_g$ and
  $\norm{V}_{L^{\infty}}$.
\end{theorem}

\begin{remark}
  The estimate \eqref{expdecay} guarantees that with a fixed buffer
  region, the error we make by restricting to a local problem decays
  exponentially away from the boundary. As the constants depend only
  on the spectral gap and the $L^{\infty}$ norm of the potential, if
  we fix a point $x$ and enlarge the buffer region $\Lambda_b$, the
  error will also decay exponentially, as long as the gap assumption
  is uniformly satisfied for the increasing buffer regions. This point
  would be further demonstrated in the numerical examples.
\end{remark}

Before we prove the theorem, let us recall the decay estimate of
Green's function from \cite{ELu:ARMA}*{Theorem 9} and its proof (see
also \cite{ELu:2013} where such estimates are used for the macroscopic
limit of Kohn-Sham density functional theory). We also remark that the
exponential decay property of the Green's function and the density
matrix also holds at the discrete level \cite{Benzi} and hence our
analysis can be also done for the discretized Hamiltonian.
\begin{prop}[Decay estimate of Green's function]\label{prop:greendecay}
  Given a Hamiltonian $H = - \Delta + V$ with $V \in L^{\infty}$.  For
  any $\lambda \not \in \spec(H)$, there exist constants
  $\gamma_{\max} > 0$ and $M$, depending only on $\dist(\lambda,
  \spec(H))$, $\abs{\lambda}$ and $\norm{V}_{L^{\infty}}$, such that
  for all $x_0$ and any $\gamma < \gamma_{\max}$, we have
  \begin{align}
    & \bigl\lVert \mc{W}_{x_0}^{-1} (\lambda - H)^{-1} \mc{W}_{x_0}
    \bigr\rVert \leq M  \\
    & \bigl\lVert \mc{W}_{x_0}^{-1} \partial_j (\lambda - H)^{-1}
    \mc{W}_{x_0} \bigr\rVert \leq M, \quad \text{for } j = 1, \ldots,
    d,
  \end{align}
  where $d$ is the dimension, and
  $\mc{W}_{x_0}$ is the multiplication operator given by
  \begin{equation}
    (\mc{W}_{x_0} f)(x) = \exp\bigl(-\gamma
    ((x - x_0)^2 + 1)^{1/2} \bigr) f(x).
  \end{equation}
\end{prop}

Applying the result to our current setting, since $\CC$ is compact and
by the gap assumption, $\dist(\CC, \spec(H)), \dist(\CC,
\spec(H_{\Lambda_d})) > e_g / 2$, the $\gamma_{\max}$ and $M$ can be
chosen for both $H$ and $H_{\Lambda_b}$ as constants depending only on
$\CC, e_g$, and $\norm{V}_{L^{\infty}}$. Moreover, the choice of the
contour only depend on the location of the spectral gap and the bottom
of the spectra of $H$ and $H_{\Lambda_b}$, which can be controlled by
$\eps_F$ and $\norm{V}_{L^{\infty}}$. Hence, the constants only depend
on $\eps_F$, $e_g$ and $\norm{V}_{L^{\infty}}$. Let us now proceed to
prove the Theorem.

\begin{proof}[Proof of Theorem~\ref{thm:accuracy}]
  Using the resolvent identity, we write the difference in density as the
  difference in operators. For $x \in \Lambda$, we have
  \begin{equation*}
    \rho(x) - \rho_{\Lambda}(x) = \frac{1}{2 \pi i} \int_{\CC}
    \bigl\langle \varphi_x, \bigl[(\lambda - H)^{-1} - (\lambda - H_{\Lambda_b})^{-1}\bigr] \varphi_x \bigr\rangle \ud \lambda
  \end{equation*}
  Since the contour $\CC$ is compact, we obtain
  \begin{equation*}
    \abs{\rho(x) - \rho_{\Lambda}(x)} \lesssim \max_{\lambda \in \CC}
    \Bigl\lvert \bigl\langle \varphi_x, \bigl[(\lambda - H)^{-1} - (\lambda - H_{\Lambda_b})^{-1}\bigr] \varphi_x \bigr\rangle \Bigr\rvert.
  \end{equation*}
  Let $\Lambda_x$ denote the support of $\varphi_x$,
  using Lemma~\ref{lem:GRI}, we have the geometric resolvent identity
  \begin{equation}\label{eq:geomresidentity}
    \begin{aligned}
      1_{\Lambda_x} \bigl[ (\lambda - H)^{-1} - (\lambda -
      H_{\Lambda_b})^{-1} \bigr] 1_{\Lambda_x}
      =  1_{\Lambda_x} (\lambda - H_{\Lambda_b})^{-1}
      [H, \Theta] (\lambda -
      H)^{-1} 1_{\Lambda_x},
    \end{aligned}
  \end{equation}
  where we take $\Theta$ such that $\Theta = 1$ in $\wt{\Lambda}_b$
  and $\Theta = 0$ outside $\Lambda_b$. The commutator $[H, \Theta]$
  can be calculated as
  \begin{equation*}
    [H, \Theta] = - (\Delta \Theta) - 2 \nabla \Theta \cdot \nabla.
  \end{equation*}
  Note that both $\Delta \Theta$ and $\nabla \Theta$ are supported on
  $\Lambda_b \backslash \wt{\Lambda}_b$ by the choice of $\Theta$.  By the
  construction of the set $\wt{\Lambda}_b$ as in \eqref{eq:wtLambdab},
  we can choose $\Theta$ such that $\norm{\Delta \Theta}_{L^{\infty}}$
  and $\norm{\nabla \Theta}_{L^{\infty}}$ are both $\Or(1)$ quantities.
  Applying \eqref{eq:geomresidentity}, we hence arrive at
  \begin{multline*}
    \Bigl\lvert\bigl\langle \varphi_x, \bigl[(\lambda - H)^{-1} -
    (\lambda - H_{\Lambda_b})^{-1}\bigr] \varphi_x \bigr\rangle
    \Bigr\rvert \leq \Bigl\lvert \bigl\langle \varphi_x, (\lambda -
    H_{\Lambda_b})^{-1} (\Delta \Theta) (\lambda - H)^{-1}
    \varphi_x \bigr\rangle \Bigr\rvert \\
    + 2 \Bigl\lvert \bigl\langle \varphi_x, (\lambda -
    H_{\Lambda_b})^{-1} (\nabla \Theta \cdot \nabla) (\lambda -
    H)^{-1} \varphi_x \bigr\rangle \Bigr\rvert.
  \end{multline*}
  The proof then concludes by estimating the two terms on the right
  hand side. These decay estimates are given by the next
  Lemma~\ref{lem:decay}.
\end{proof}

\begin{lemma}[Decay estimates]\label{lem:decay}
  Let $f$ be a $L^{\infty}$ function such that
  \begin{equation*}
    \supp f \subset \Lambda_b \backslash \wt{\Lambda}_b.
  \end{equation*}
  There exist constants $\gamma_{\max} > 0$ and $C$ such that for any
  $\lambda \in \CC$ and $\gamma < \gamma_{\max}$, we have
  \begin{align}
    & \Bigl\lvert \bigl\langle \varphi_x, (\lambda -
    H_{\Lambda_b})^{-1} f (\lambda - H)^{-1}
    \varphi_x \bigr\rangle \Bigr\rvert \leq C \exp(- 2 \gamma (\dist(x, \Lambda_b^c) - 1)) \norm{f}_{L^{\infty}}; \\
    & \label{eq:derivdecaybound} \Bigl\lvert \bigl\langle \varphi_x,
    (\lambda - H_{\Lambda_b})^{-1} f \partial_j (\lambda - H)^{-1}
    \varphi_x \bigr\rangle \Bigr\rvert \leq C \exp(- 2 \gamma
    (\dist(x, \Lambda_b^c) - 1)) \norm{f}_{L^{\infty}}.
  \end{align}
  where $f$ is interpreted as a multiplication operator on the left
  hand sides.
\end{lemma}

\begin{proof}
  By inserting the exponential weight $\mc{W}_x$ centered at $x$, we
  can estimate
  \begin{equation*}
    \begin{aligned}
      \Bigl\lvert \bigl\langle \varphi_x, (\lambda -
      H_{\Lambda_b})^{-1} & f (\lambda - H)^{-1} \varphi_x
      \bigr\rangle
      \Bigr\rvert \\
      & = \Bigl\lvert \bigl\langle (\lambda - H_{\Lambda_b})^{-1}
      \varphi_x, f (\lambda - H)^{-1} \varphi_x \bigr\rangle
      \Bigr\rvert \\
      & = \Bigl\lvert \bigl\langle \mc{W}_x^{-1}(\lambda -
      H_{\Lambda_b})^{-1}\mc{W}_x \mc{W}_x^{-1}\varphi_x, \mc{W}_x f
      \mc{W}_x \mc{W}_x^{-1} (\lambda - H)^{-1} \mc{W}_x \mc{W}_x^{-1}
      \varphi_x  \bigr\rangle \Bigr\rvert \\
      & \leq \norm{\mc{W}_x^{-1} \varphi_x}_{L^2}^2
      \norm{\mc{W}_x^{-1} (\lambda - H_{\Lambda_b})^{-1} \mc{W}_x}
      \norm{\mc{W}_x^{-1} (\lambda - H)^{-1} \mc{W}_x} \norm{\mc{W}_x f
        \mc{W}_x} \\
      & \lesssim \norm{\mc{W}_x f \mc{W}_x},
    \end{aligned}
  \end{equation*}
  where the last inequality uses Proposition~\ref{prop:greendecay} for
  the operators $H$ and $H_{\Lambda_b}$.  Note that $\mc{W}_x f
  \mc{W}_x$ is a multiplication operator
  \begin{equation*}
    \bigl((\mc{W}_x f \mc{W}_x) u\bigr)(y) =
    \exp\bigl(- 2 \gamma ( ( x - y)^2 + 1)^{1/2} \bigr) f(y) u(y),
  \end{equation*}
  Hence,
  \begin{equation*}
    \begin{aligned}
      \norm{\mc{W}_x f \mc{W}_x} & = \bigl\lVert \exp\bigl(- 2
      \gamma ( ( x - \cdot)^2 + 1)^{1/2} \bigr)
      f(\cdot) \bigr\rVert_{L^{\infty}} \\
      & \leq \exp\bigl(- 2 \gamma ( \dist(x, \supp f)^2 +
      1)^{1/2} \bigr) \norm{f}_{L^{\infty}} \\
      & \leq \exp\bigl(- 2 \gamma (\dist(x, \Lambda_b^c) - 1) \bigr)
      \norm{f}_{L^{\infty}}.
      \end{aligned}
  \end{equation*}
  The proof of \eqref{eq:derivdecaybound} is analogous and will be
  omitted.
\end{proof}

\section{Gap assumption on the subsystem}\label{gap}

In this section, we validate the  DAC algorithm and
also our analytical results by numerical examples.  By several
examples in one dimension and two dimensions, we show the accuracy of
the subsystem if \assumpref{assump} is valid.  Moreover, the loss of
accuracy for the subsystem is observed if \assumpref{assump} fails.
While in practice, we do not have easy criteria of selection of
subdomains that guarantees \eqref{assumpa}--\eqref{assumpc}, numerical
results show that they are essential for the accuracy of the method.

\begin{example}
\label{ex:insulator1d}
Consider an infinite array of atoms on a line with $X_i=i$, for
$i\in\mathbb{Z}$.  Each atom has one valence electron and spin
degeneracy is ignored.  We adopt an example from \cite{Glxe}, where
$V$ is chosen with the following form
\[
V(x)= -\sum_{i\in\mathbb{Z}}\frac{a}{\sqrt{2\pi\sigma^2}}\exp{[-(x-X_i)^2/2\sigma^2]}.
\]

Figure~\ref{fig:bandstructure} shows band structures when $a=5$,
$\sigma=0.15$, and $a=5$, $\sigma=0.45$. We will assume one electron
per atom (note that spin degeneracy is ignored).  As studied in
\cite{Glxe}, the gap is very small when $a=5$, $\sigma=0.45$ and the
system behaves like a metal.  Therefore, it is clear that for selected
parameters, the corresponding system has a gap in the spectrum
(insulator) as in Figure~\ref{fig:bandstructureinsulator}, while it
does not have a gap (metal) as in Figure~\ref{fig:bandstructuremetal}.
In other words, \eqref{assumpa}--\eqref{assumpb} are valid in
Figure~\ref{fig:bandstructureinsulator}, and invalid in
Figure~\ref{fig:bandstructuremetal}.

\begin{figure}[htbp]
\vspace{-6em}
\centering
\subfigcapskip -0.6in
\subfigure[Insulator]{%
\label {fig:bandstructureinsulator}
\includegraphics[width=2.5in]{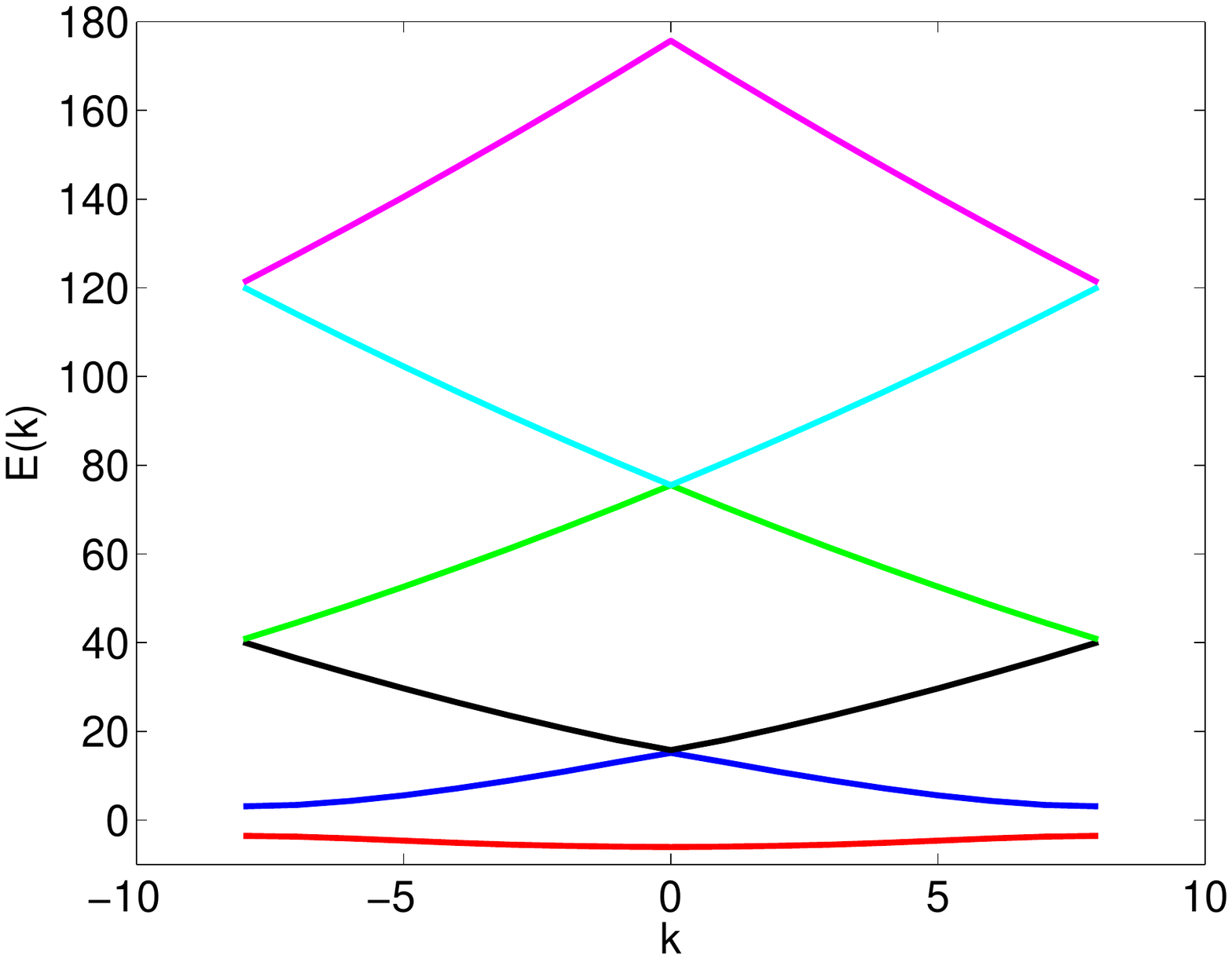}}%
\subfigure[Metal]{
\label {fig:bandstructuremetal}
\includegraphics[width=2.5in]{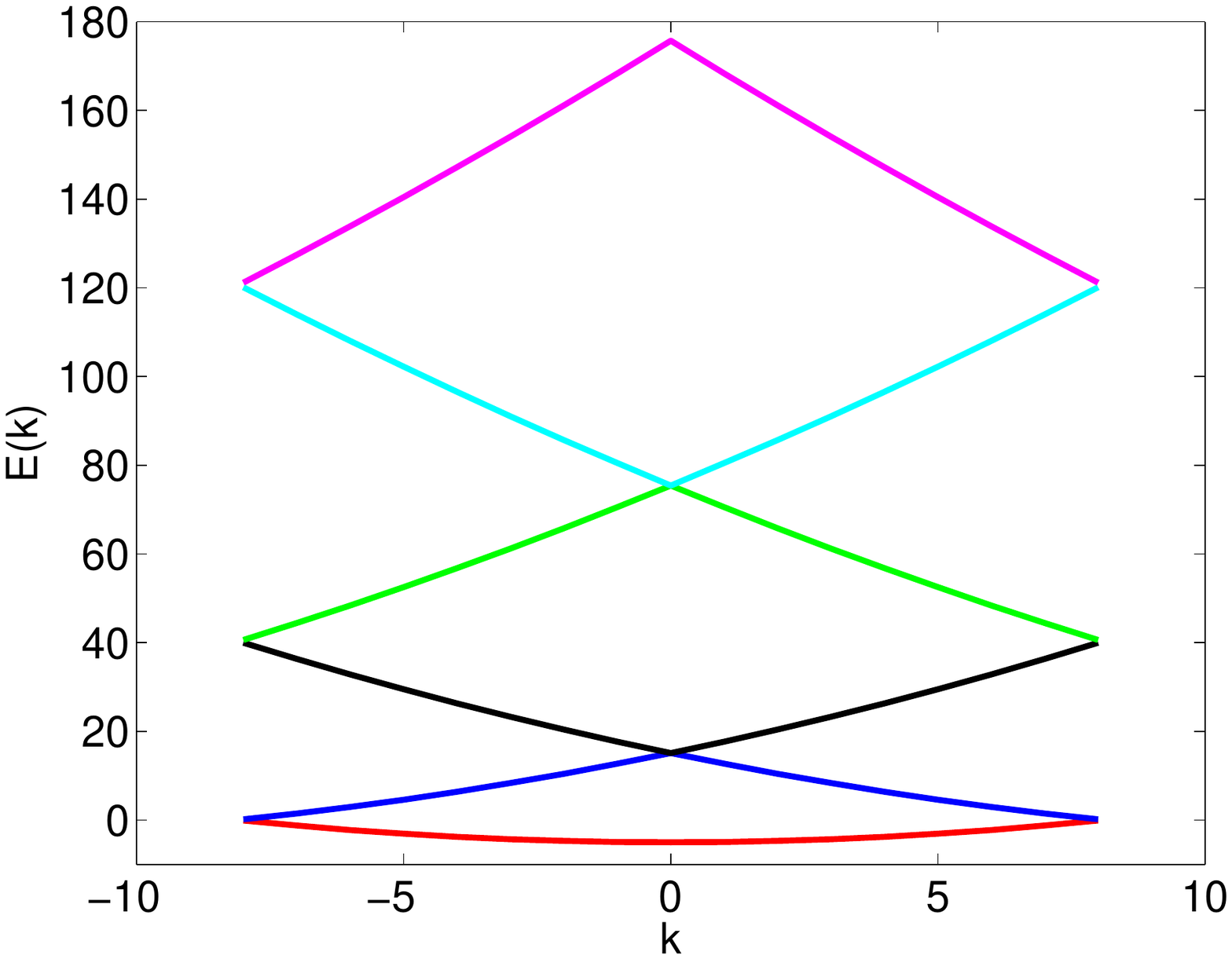}}%
\vspace{-2em}
\caption{\small Band structures for different parameters in
  \exref{ex:insulator1d}. The first band is occupied. (a) Insulator,
  where $a=5$ and $\sigma=0.15$; (b) Metal, where $a=5$ and
  $\sigma=0.45$.}\label {fig:bandstructure}
\end{figure}

Choose $\Om=\mathbb{R}$ and $\Lam_b = [0, 16]$. From the left column of Figure \ref{fig:insulator}, one can see \eqref{assumpc} is valid for $H_{\Lam_b}$ with three different boundary conditions, including Dirichlet boundary condition (DBC), Neumann boundary condition (NBC), and periodic boundary condition (PBC).
\begin{figure}[htbp]
\vspace{-6em}
\centering
\subfigcapskip -0.8in
\subfigure[Energy level with DBC]{%
\label {fig:insulatorenergyleveldbc}
\includegraphics[width=2.5in]{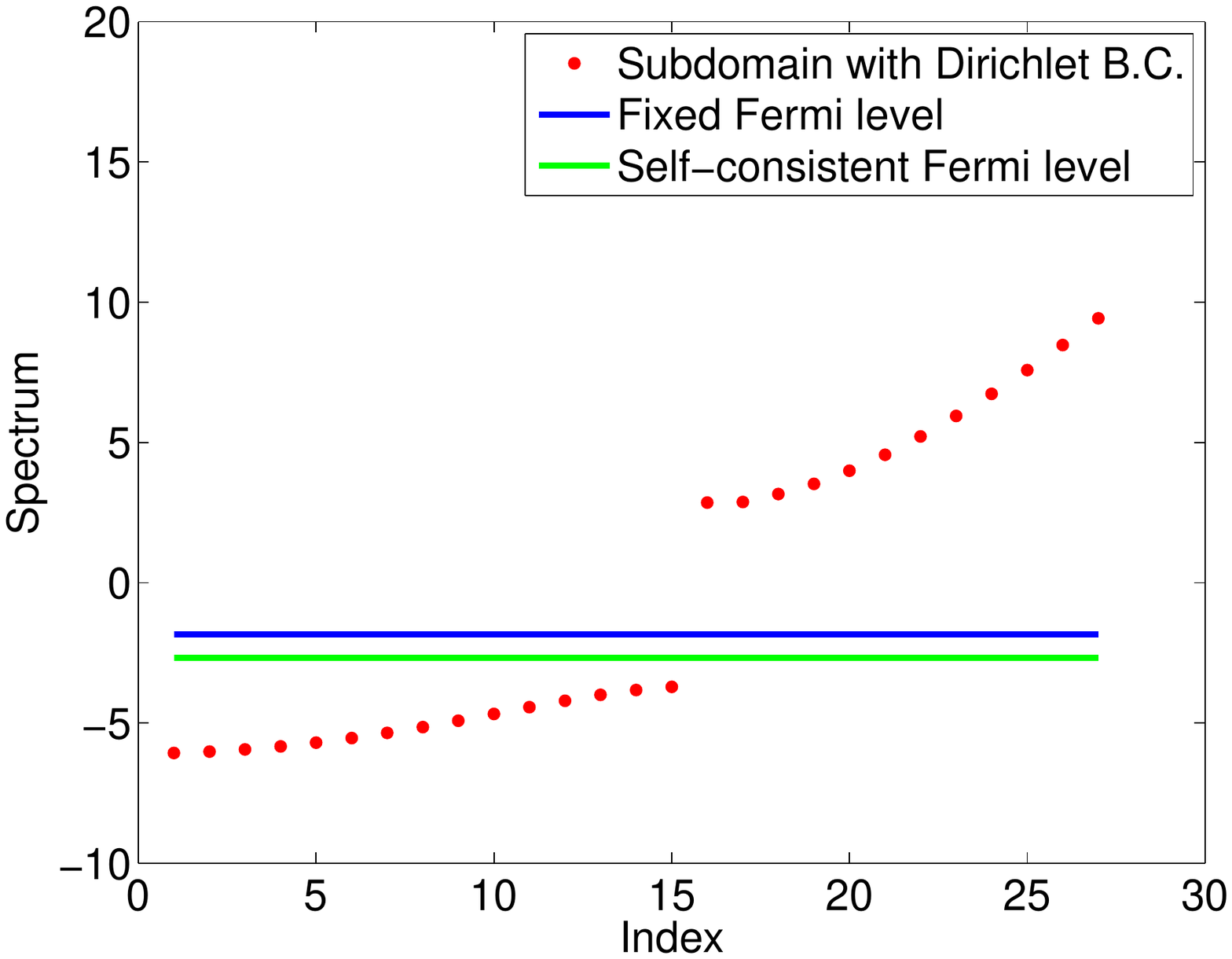}}%
\subfigure[Density difference with DBC]{%
\label {fig:insulatordensitydiffdbc}
\includegraphics[width=2.5in]{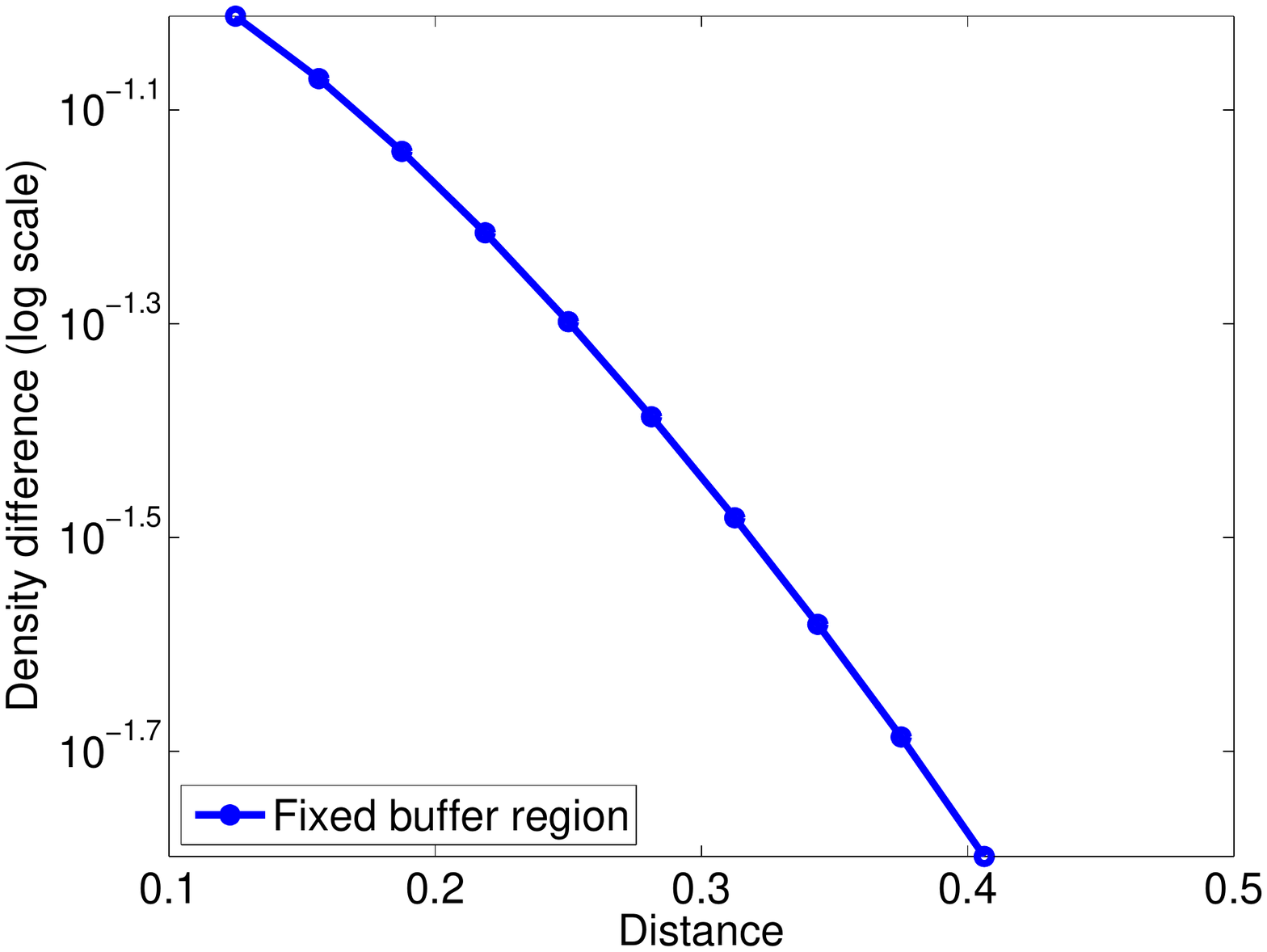}}%
\\
\vspace{-8em}
\subfigure[Energy level with NBC]{
\label {fig:insulatorenergylevelnbc}
\includegraphics[width=2.5in]{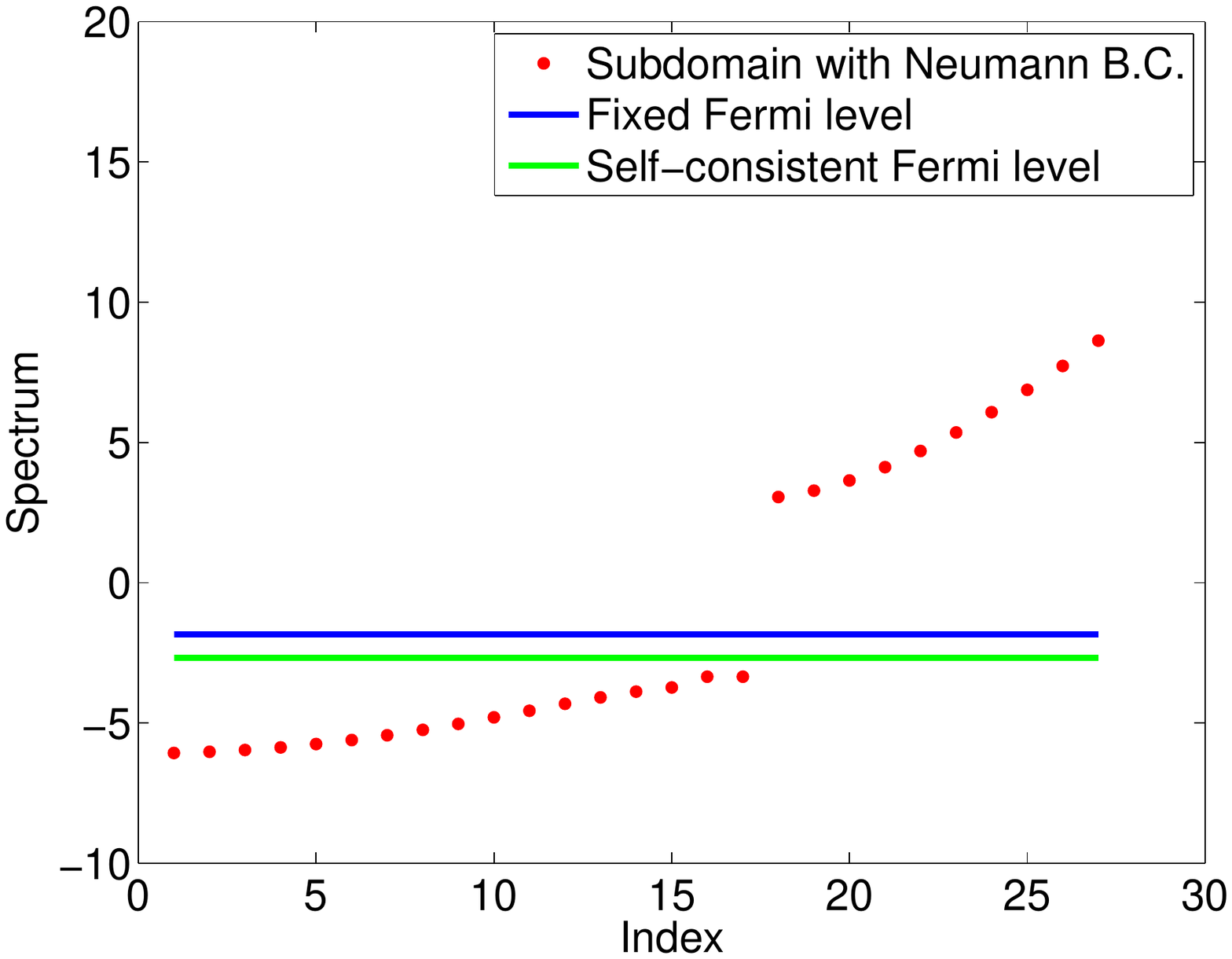}}%
\subfigure[Density difference with NBC]{
\label {fig:insulatordensitydiffnbc}
\includegraphics[width=2.5in]{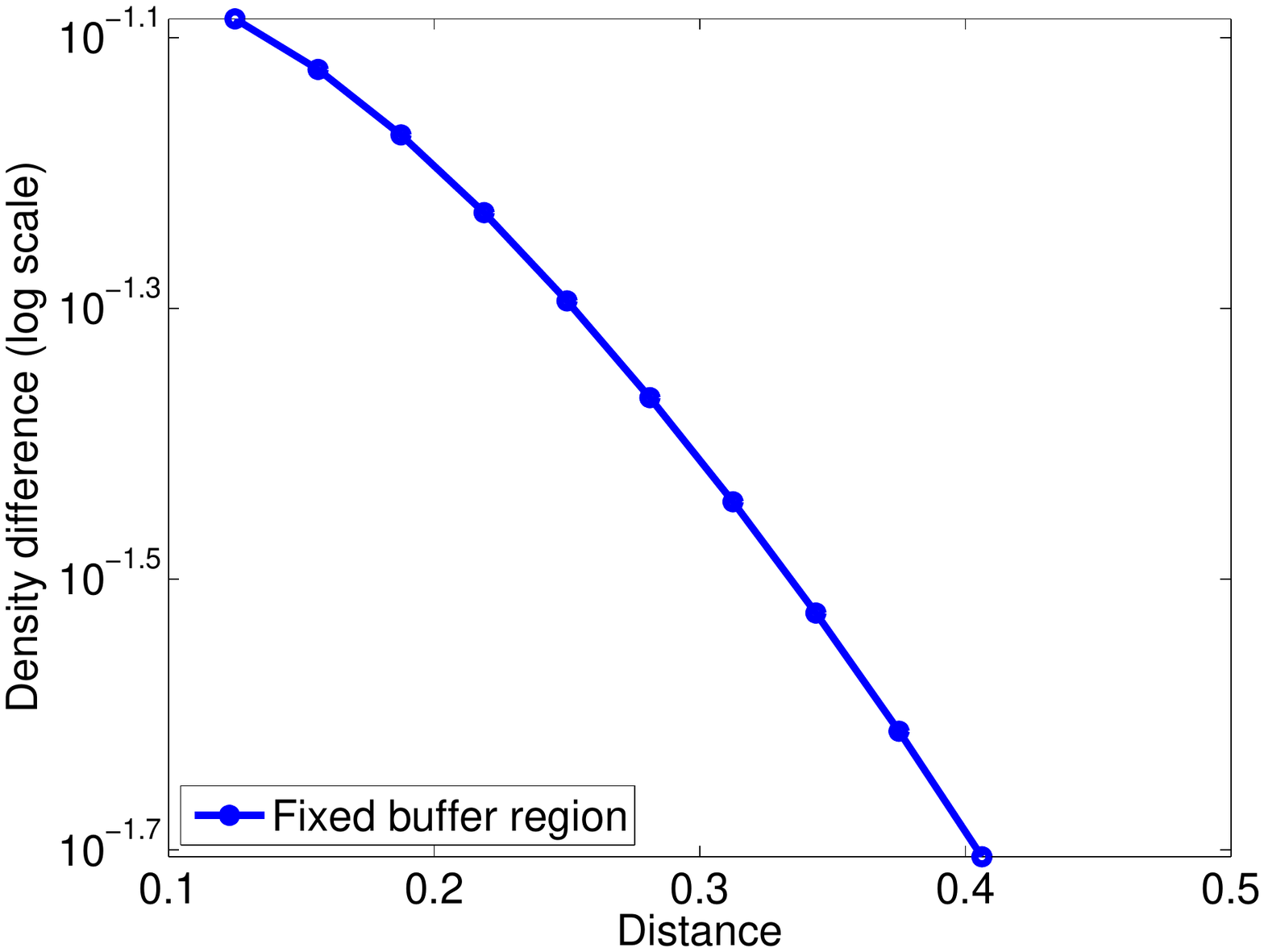}}%
\\
\vspace{-8em}
\subfigure[Energy level with PBC]{
\label {fig:insulatorenergylevelpbc}
\includegraphics[width=2.5in]{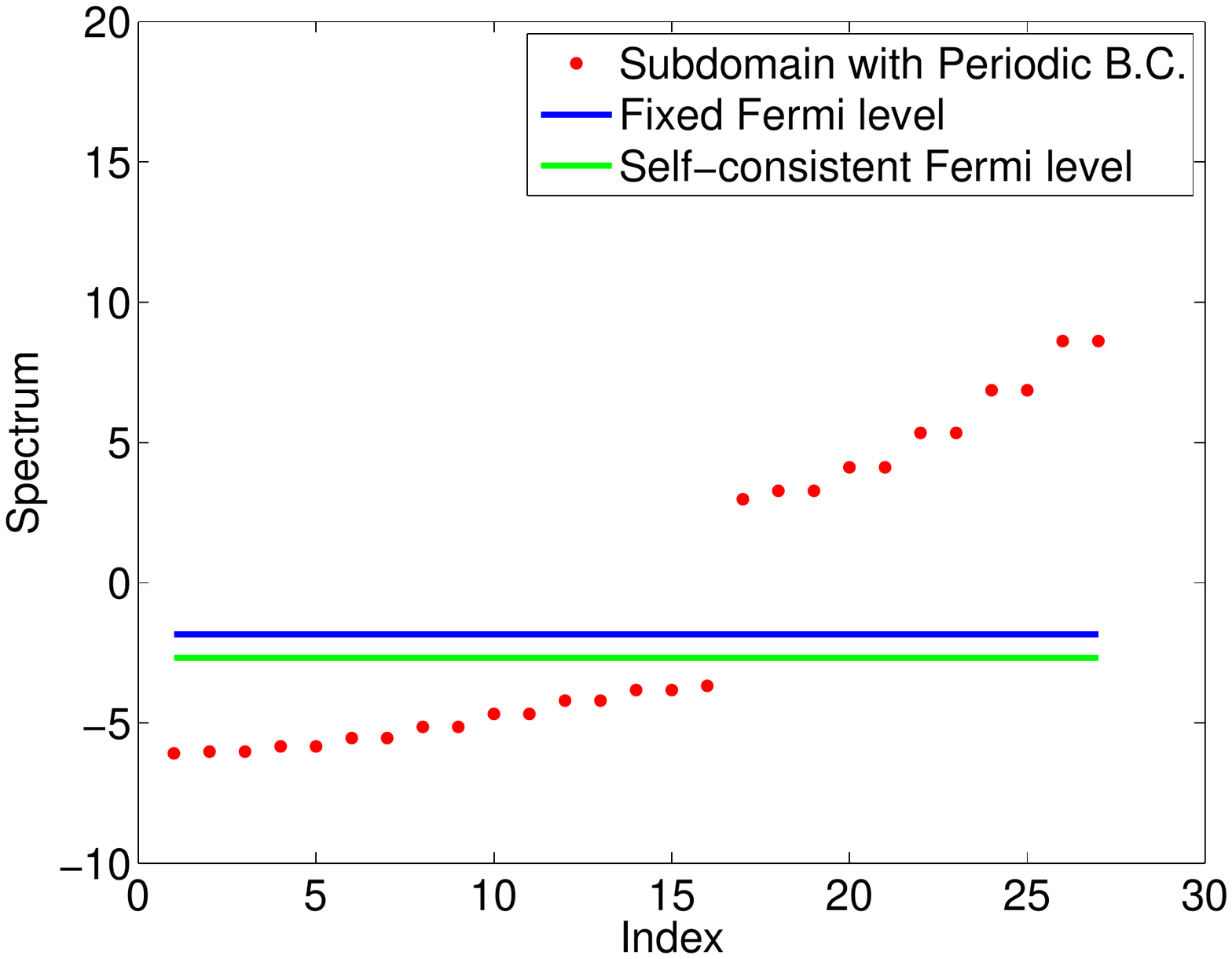}}%
\subfigure[Density difference with PBC]{
\label {fig:insulatordensitydiffpbc}
\includegraphics[width=2.5in]{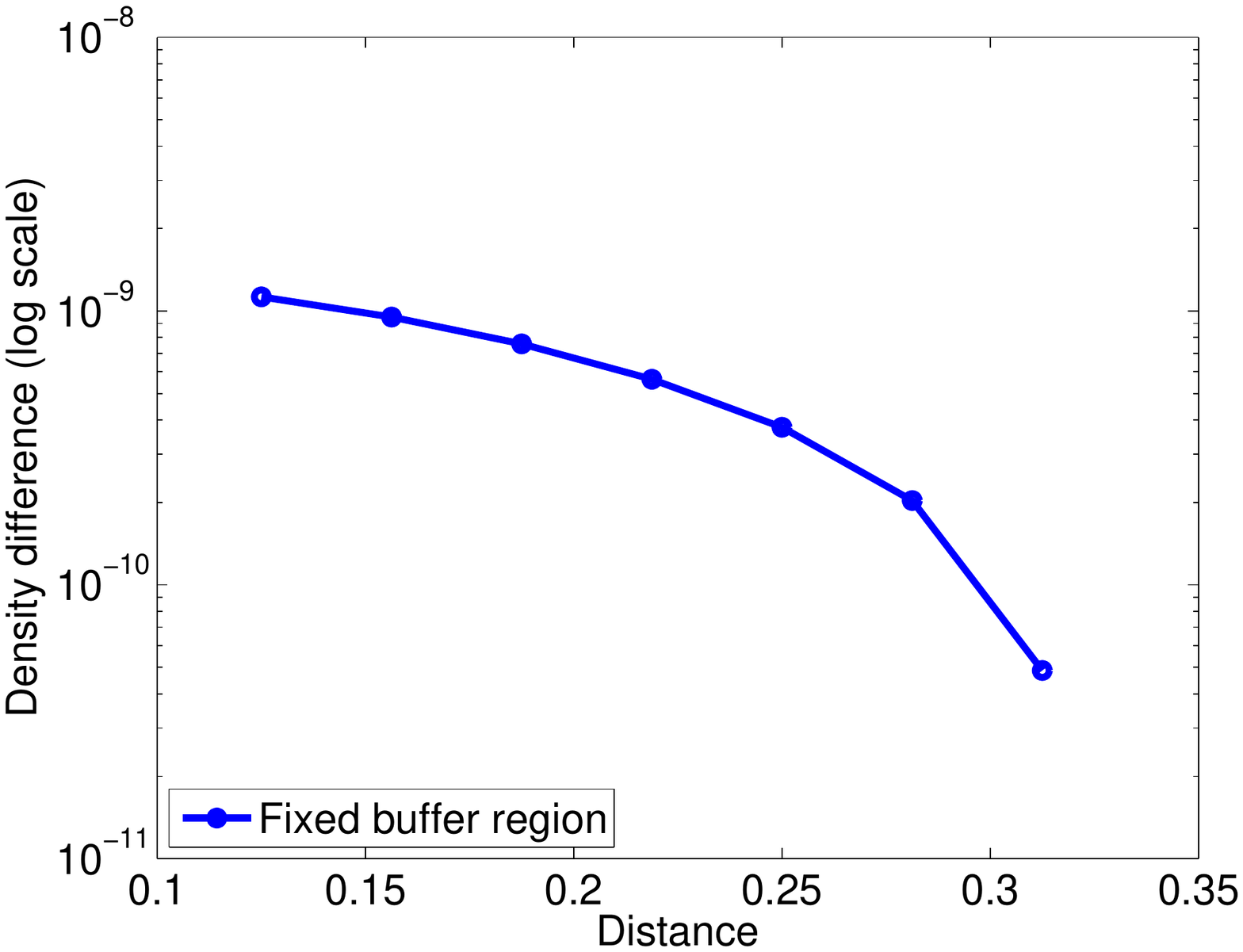}}%
\vspace{-4em}
\caption{\small Energy levels of the subsystem, and
  $\abs{\rho(x)-\rho_{\Lam}(x)}$ as a function of $x$ for $\Lam=[0.1,
  15.9]$ and $\Lam_b=[0,16]$ with different boundary conditions in
  \exref{ex:insulator1d}. (a) Energy level with DBC; (b) Density
  difference with DBC; (c) Energy level with NBC; (b) Density
  difference with NBC; (e) Energy level with PBC; (f) Density
  difference with PBC.  In the left column, red dots denote energy
  levels of the subsystem, blue line denotes the fixed Fermi level
  $\eps_F=(\eps_{\text{occ}}+\eps_{\text{unocc}})/2$, and green line
  denotes Fermi level obtained by the DAC method in a self-consistent
  manner, respectively. In the right column, density difference is
  plotted in the log scale and decays exponentially, which verifies
  \eqref{expdecay}.}\label {fig:insulator}
\end{figure}
Fix $\eps_F=(\eps_{\text{occ}}+\eps_{\text{unocc}})/2$ and $\Lam=[0.1,
15.9]$.  Note that $\dist(\Lam,\Lam_b^c)\ge0.1$, satisfying the
assumption on the subdomain and the buffer region. This condition
holds true for all examples in the section.  We compare
$\abs{\rho(x)-\rho_{\Lam}(x)}$ as a function of $x$ for three boundary
conditions in the right column of Figure~\ref{fig:insulator}.  Density
differences are plotted in the log scale and decay exponentially,
which verifies \eqref{expdecay}.  Quantitatively,
the method has the best performance
when PBC is used. Moreover, for
the self-consistent Fermi level, density differences behave in the
same manner.

Furthermore, for $\Lam = [0, 16]$ and a series of enlarged buffer regions $\Lam_b = [0-x, 16+x]$ ($x\ge 0.1$), we compare $\abs{\rho(0)-\rho_{\Lam}(0)}$ as a function of $x$ in Figure \ref{fig:insulatorbuffer}. Density differences are plotted in the log scale and decay exponentially, since \eqref{assumpc} is valid for the series of $\Lam_b$.
\begin{figure}[htbp]
\vspace{-4em}
\centering
\subfigcapskip -0.5in
\subfigure[Density difference with DBC]{%
\label {fig:insulatordensitydiffdbcbuffer}
\includegraphics[width=2.0in]{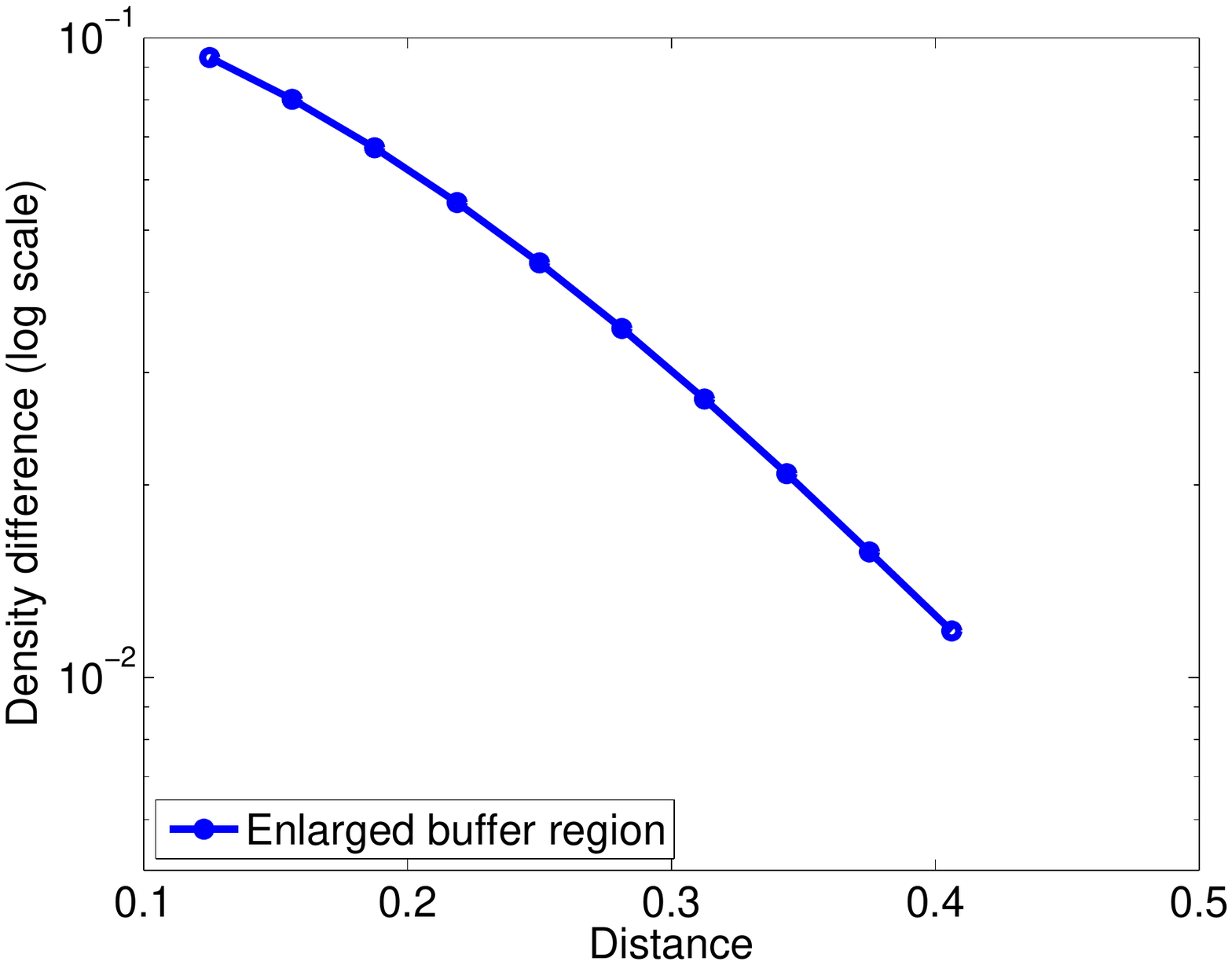}}%
\subfigure[Density difference with NBC]{
\label {fig:insulatordensitydiffnbcbuffer}
\includegraphics[width=2.0in]{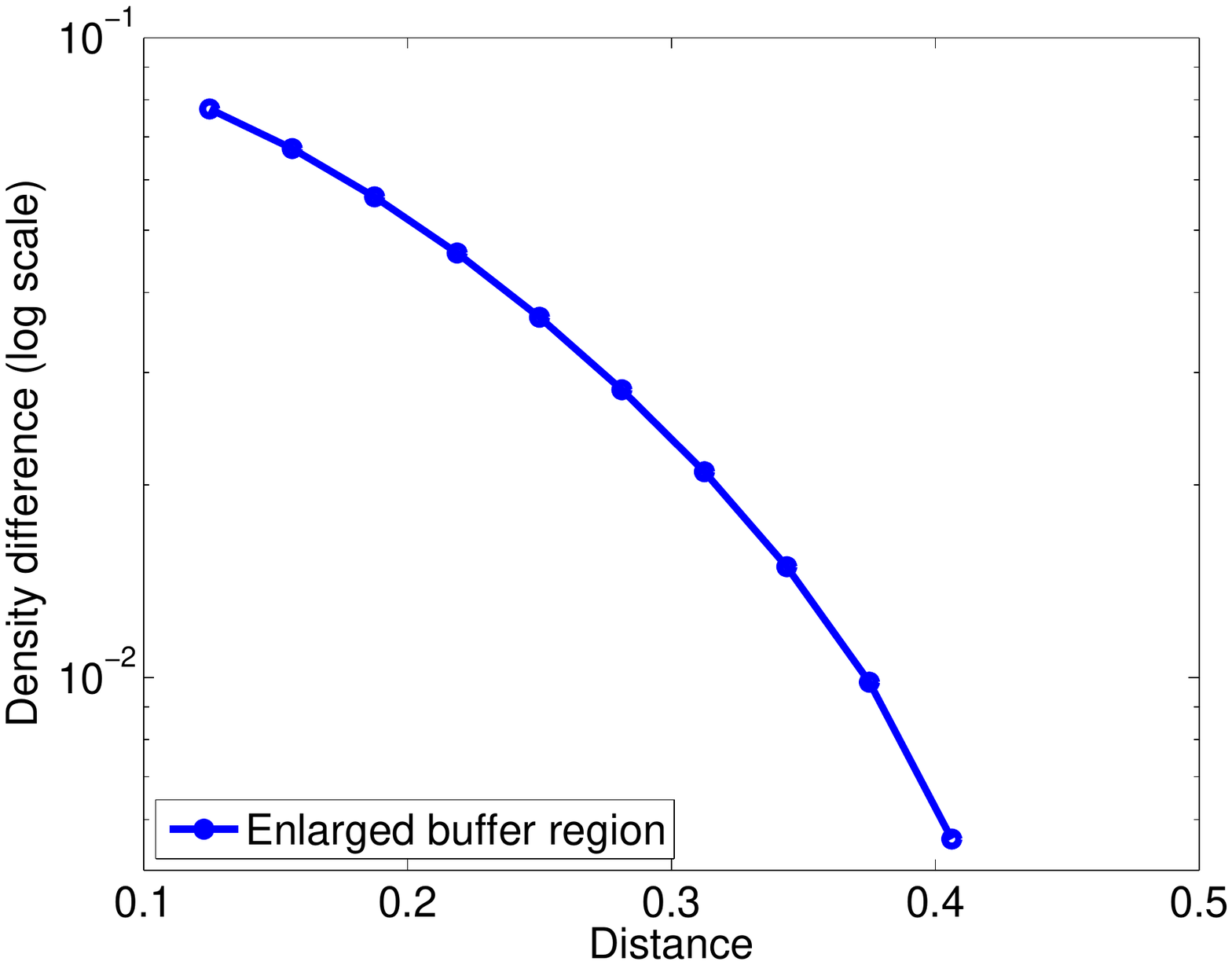}}%
\subfigure[Density difference with PBC]{
\label {fig:insulatordensitydiffpbcbuffer}
\includegraphics[width=2.0in]{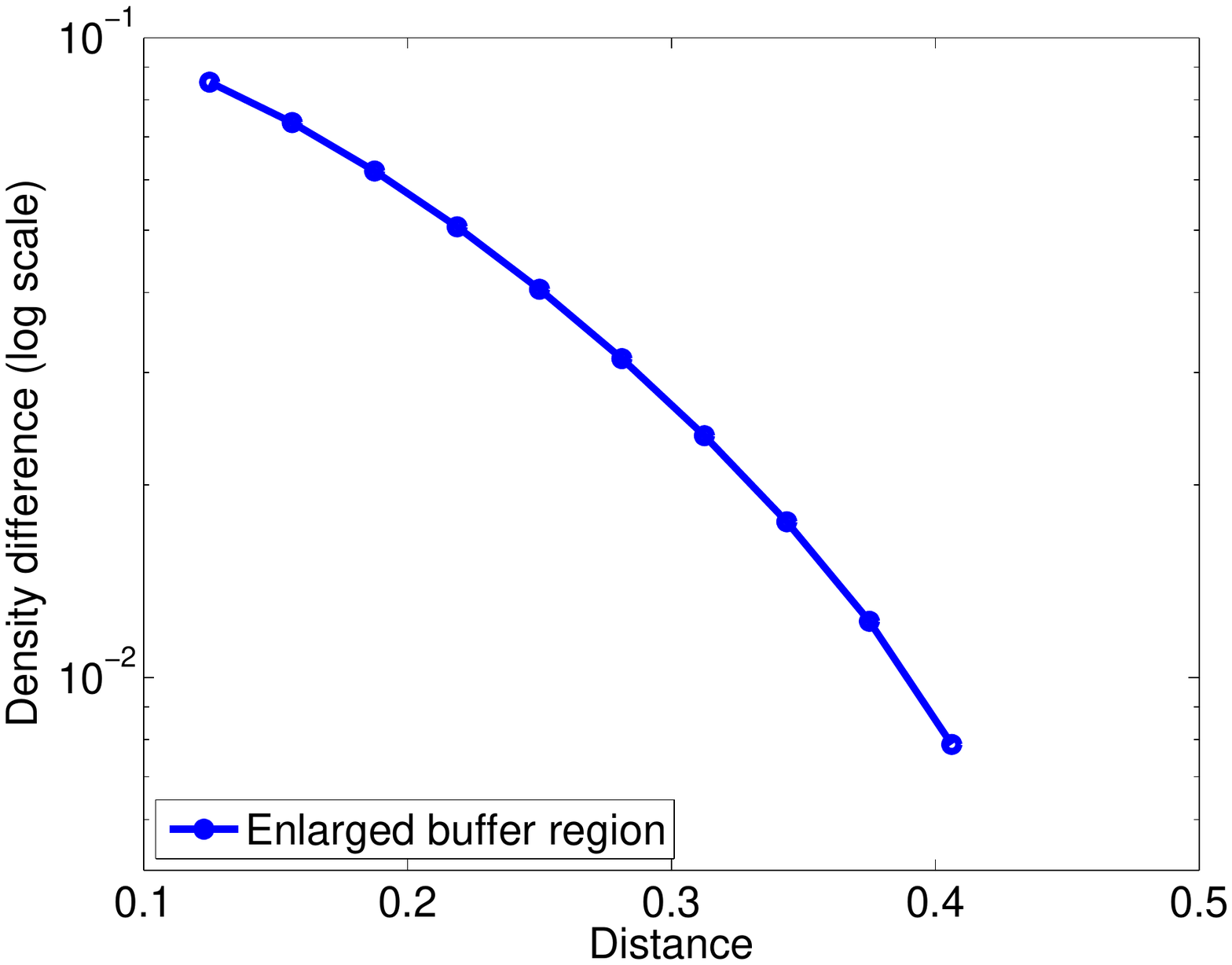}}%
\vspace{-2em}
\caption{\small  $\abs{\rho(0)-\rho_{\Lam}(0)}$ as a function of $x$  for $\Lam = [0, 16]$ and a series of enlarged buffer regions $\Lam_b = [0-x, 16+x]$ ($x\ge 0.1$) with different boundary conditions in \exref{ex:insulator1d}. (a) Density difference with DBC; (b) Density difference with NBC;
(c) Density difference with PBC.
Density difference is plotted in the log scale. Exponential decay rates are observed in all cases since
\assumpref{assump} is valid.}\label {fig:insulatorbuffer}
\end{figure}

\end{example}

\begin{example}
\label{ex:metal1d}
We now choose $a=5$ and $\sigma=0.45$ such that
\eqref{assumpa}--\eqref{assumpb} are invalid as shown in
Figure~\ref{fig:bandstructuremetal}. Set $\Lam_b = [0, 16]$ and $\Lam
= [0.1, 15.9]$.  From the left column of Figure~\ref{fig:metal}, one
can see \eqref{assumpc} is invalid for $H_{\Lam_b}$ with all three
boundary conditions.
\begin{figure}[htbp]
\vspace{-6em}
\centering
\subfigcapskip -0.8in
\subfigure[Energy level with DBC]{%
\label {fig:metalenergyleveldbc}
\includegraphics[width=2.5in]{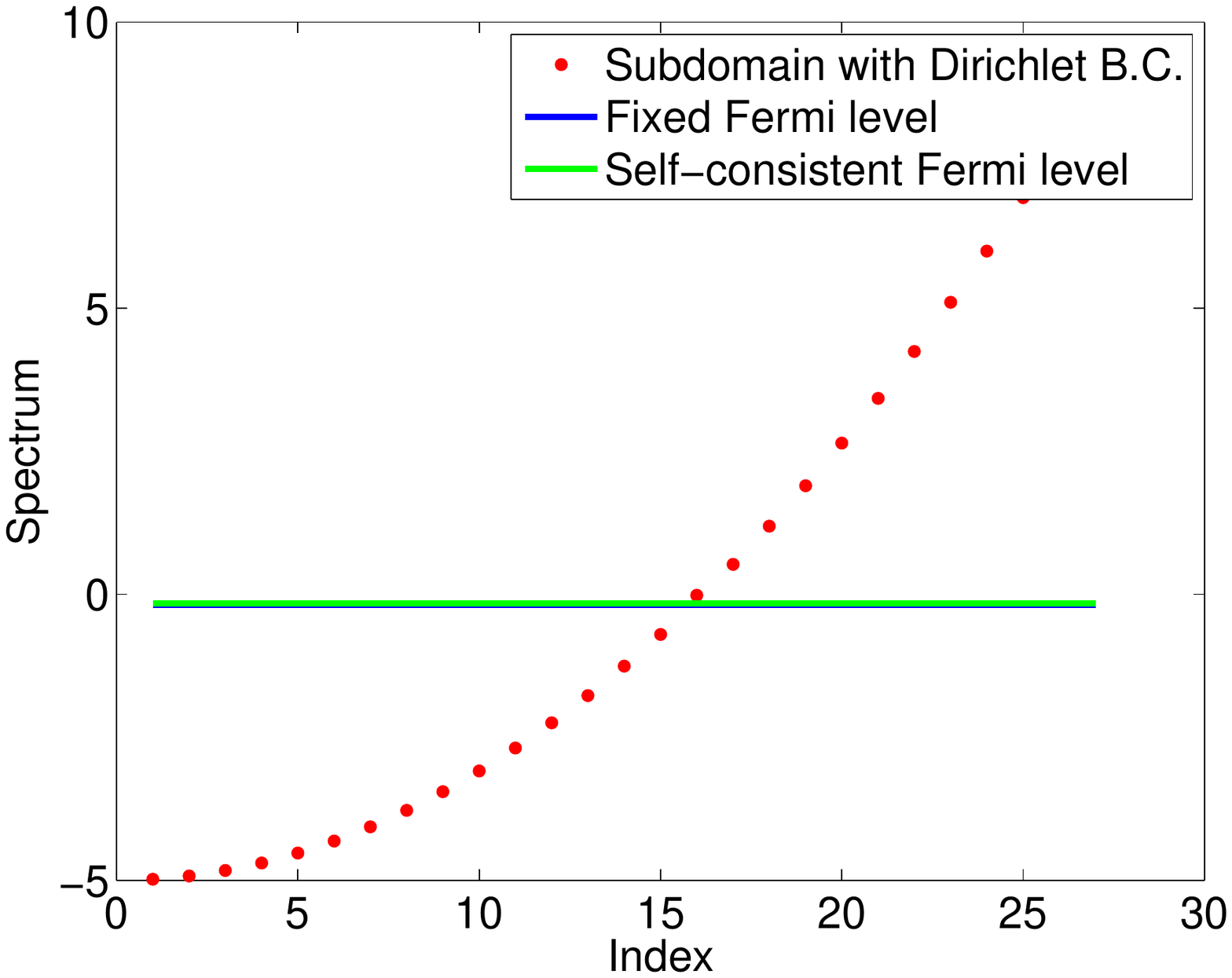}}%
\subfigure[Density difference with DBC]{%
\label {fig:metaldensitydiffdbc}
\includegraphics[width=2.5in]{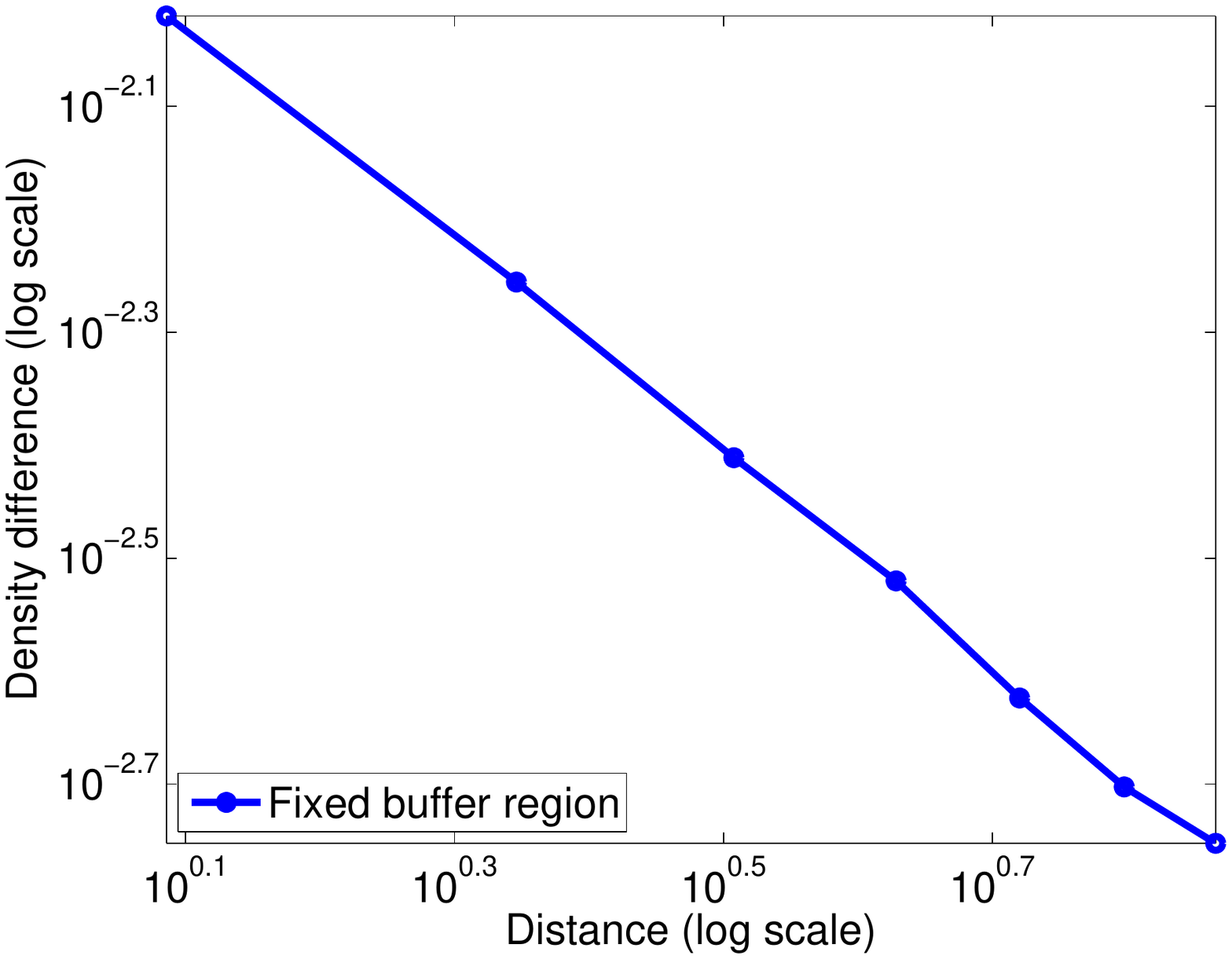}}%
\\
\vspace{-8em}
\subfigure[Energy level with NBC]{
\label {fig:metalenergylevelnbc}
\includegraphics[width=2.5in]{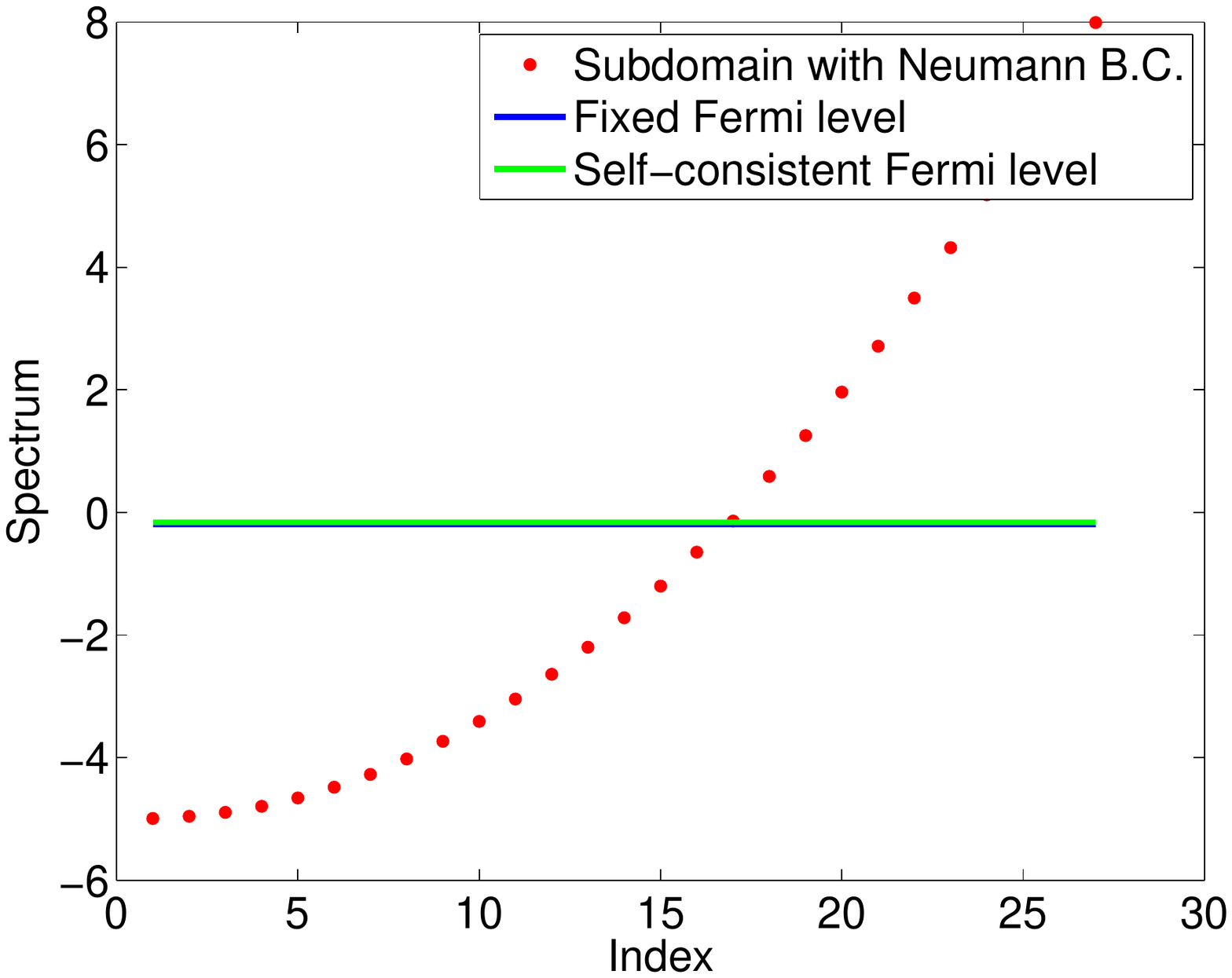}}%
\subfigure[Density difference with NBC]{
\label {fig:metaldensitydiffnbc}
\includegraphics[width=2.5in]{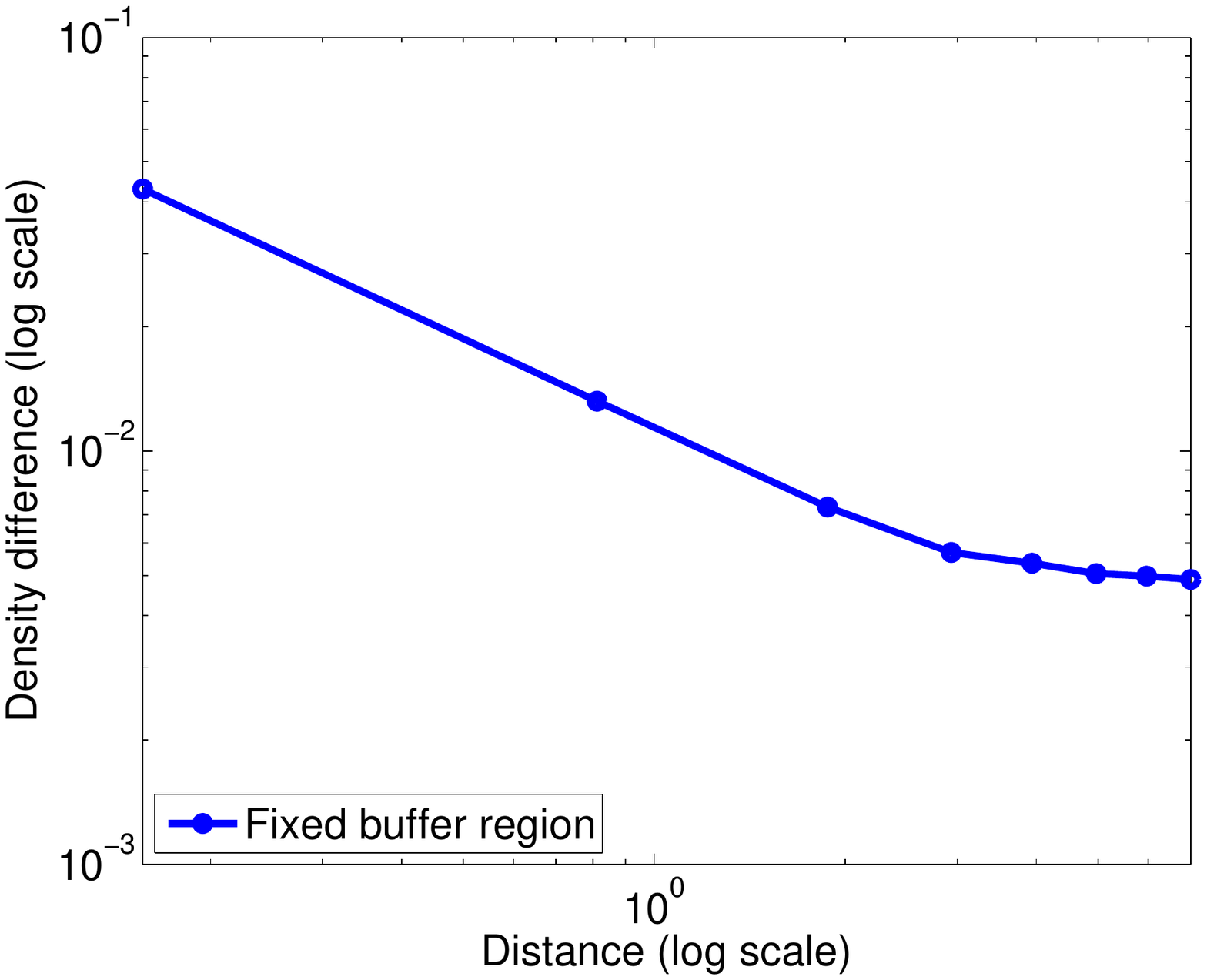}}%
\\
\vspace{-8em}
\subfigure[Energy level with PBC]{
\label {fig:metalenergylevelpbc}
\includegraphics[width=2.5in]{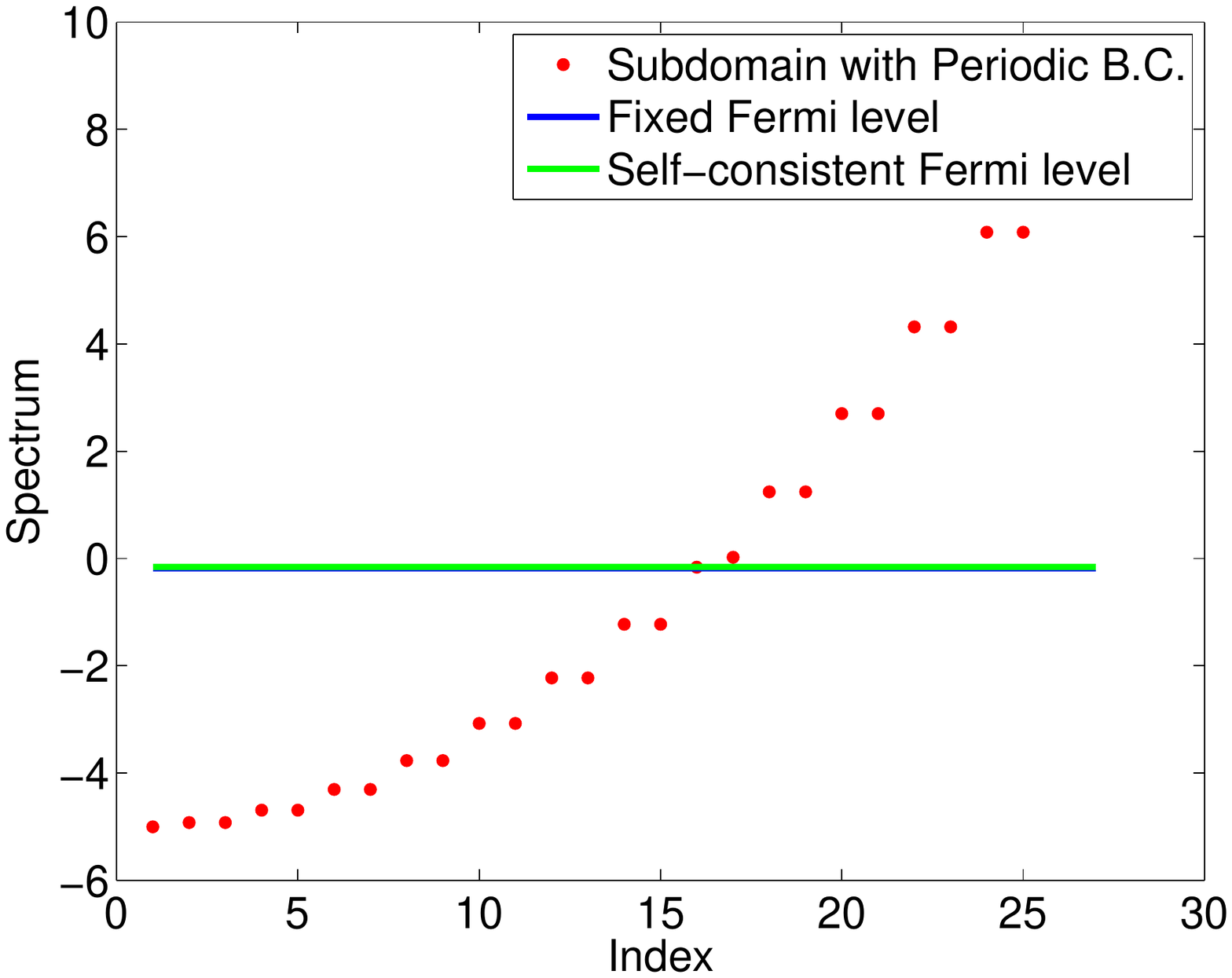}}%
\subfigure[Density difference with PBC]{
\label {fig:metaldensitydiffpbc}
\includegraphics[width=2.5in]{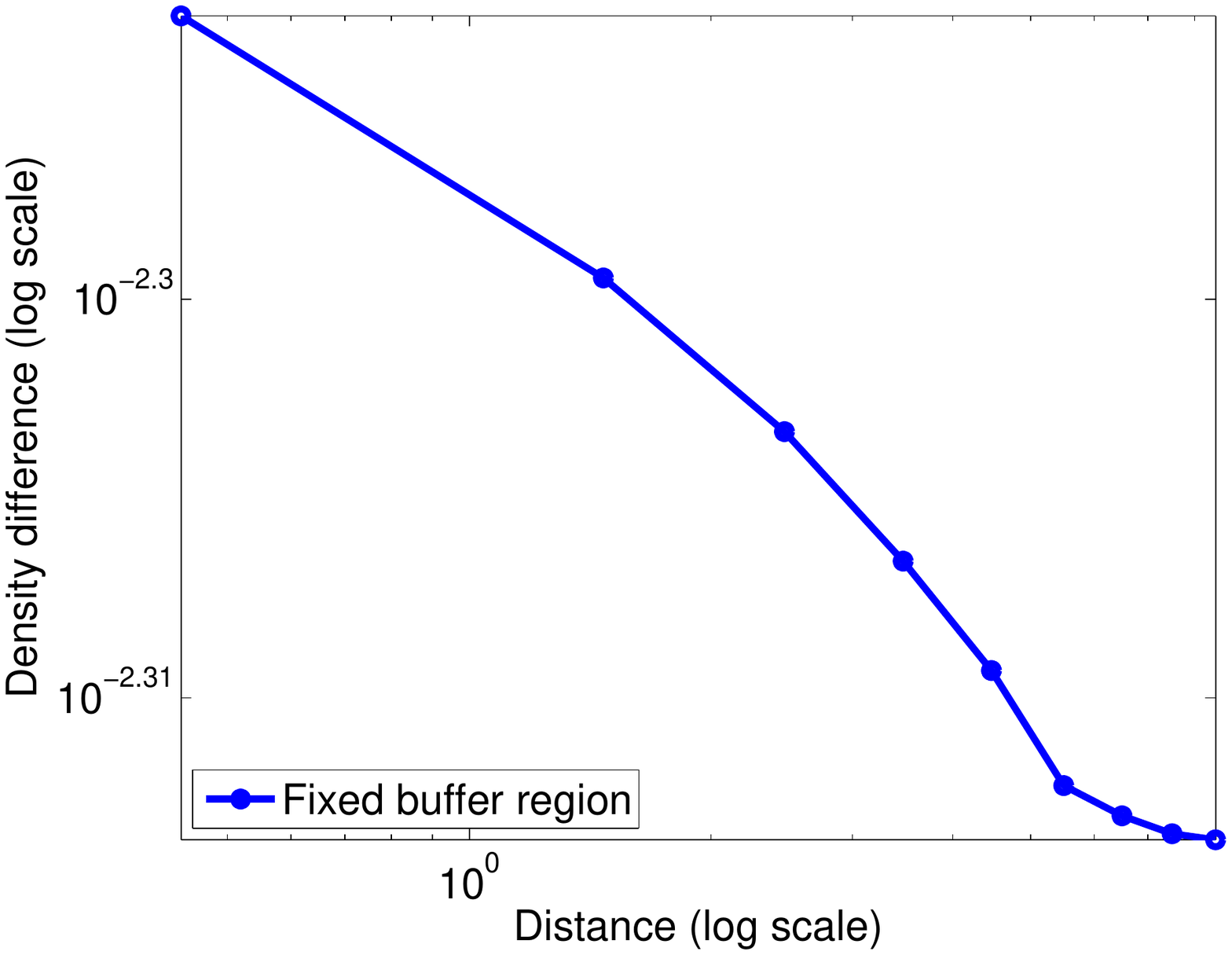}}%
\vspace{-4em}
\caption{\small  Energy levels of the subsystem, and $\abs{\rho(x)-\rho_{\Lam}(x)}$ as a function of $x$ in the log-log scale for $\Lam = [0.1, 15.9]$ and $\Lam_b = [0, 16]$ with different boundary conditions in \exref{ex:metal1d}.  (a) Energy level with DBC; (b) Density difference with DBC; (c) Energy level with NBC; (b) Density difference with NBC; (e) Energy level with PBC; (f) Density difference with PBC.
In the left column, red dots denote energy levels of the subsystem, blue line denotes the fixed Fermi level
$\eps_F=(\eps_{\text{occ}}+\eps_{\text{unocc}})/2$, and green line denotes Fermi level
obtained by the DAC method in a self-consistent manner, respectively. Only algebraic decay rate is observed
due the failure of \assumpref{assump}.}\label {fig:metal}
\end{figure}
Fix $\eps_F=(\eps_{\text{occ}}+\eps_{\text{unocc}})/2$. We compare
$\abs{\rho(x)-\rho_{\Lam}(x)}$ as a function of $x$ in the log-log
scale for three boundary conditions in the right column of Figure
\ref{fig:metal}.  Density differences decay algebraically since
\assumpref{assump} fails. Quantitatively, the method has the best performance
when PBC is used.

Furthermore, for $\Lam=[0, 16]$ and a series of enlarged buffer regions $\Lam_b=[0-x, 16+x]$ ($x\ge 0.1$), we compare $\abs{\rho(0)-\rho_{\Lam}(0)}$ as a function of $x$ in the log-log scale in Figure \ref{fig:metalbuffer}. Density differences decay algebraically as a result of the invalidity of \eqref{assumpc} for the series of $\Lam_b$.
\begin{figure}[htbp]
\vspace{-4em}
\centering
\subfigcapskip -0.5in
\subfigure[Density difference with DBC]{%
\label {fig:metaldensitydiffdbcbuffer}
\includegraphics[width=2.0in]{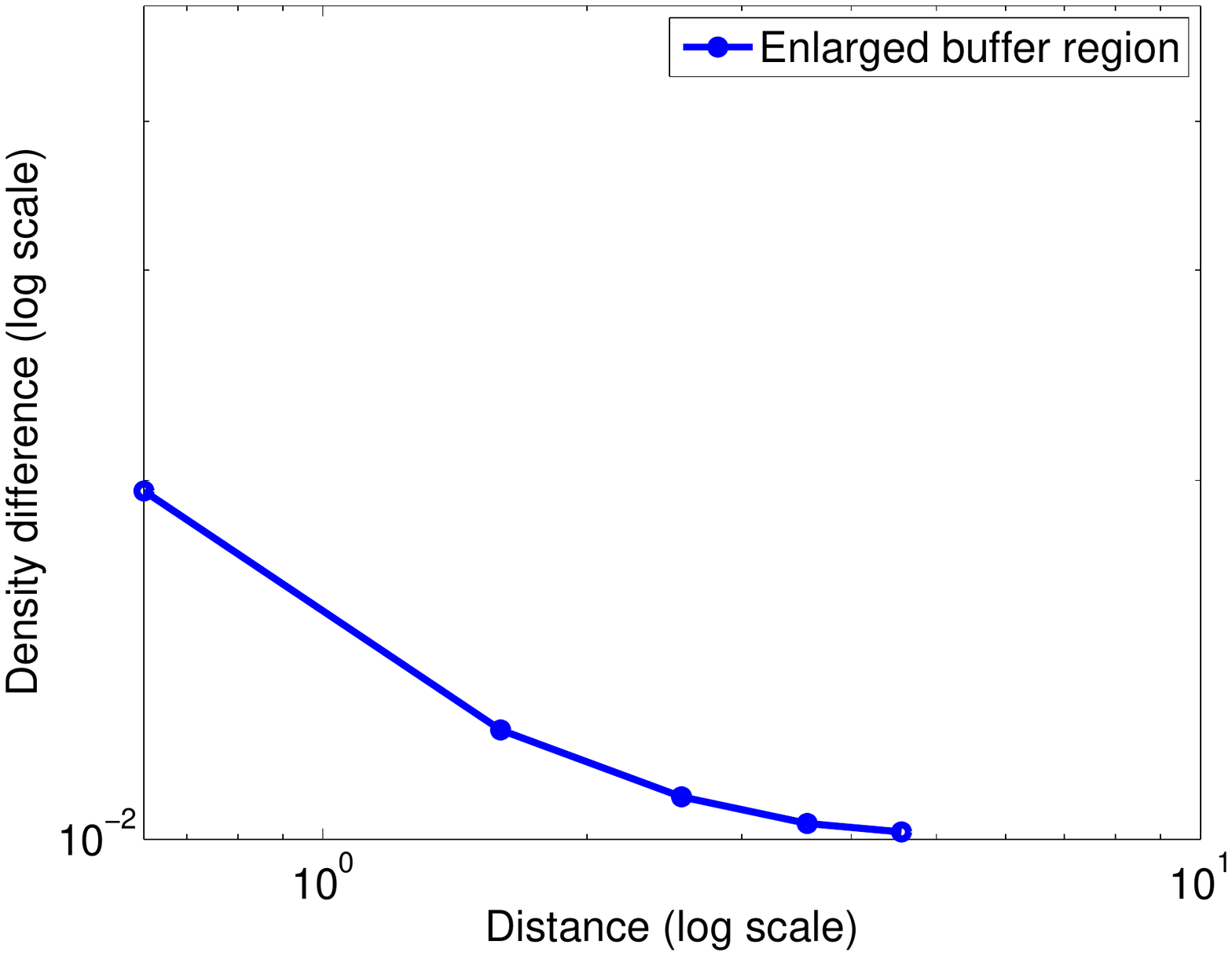}}%
\subfigure[Density difference with NBC]{
\label {fig:metaldensitydiffnbcbuffer}
\includegraphics[width=2.0in]{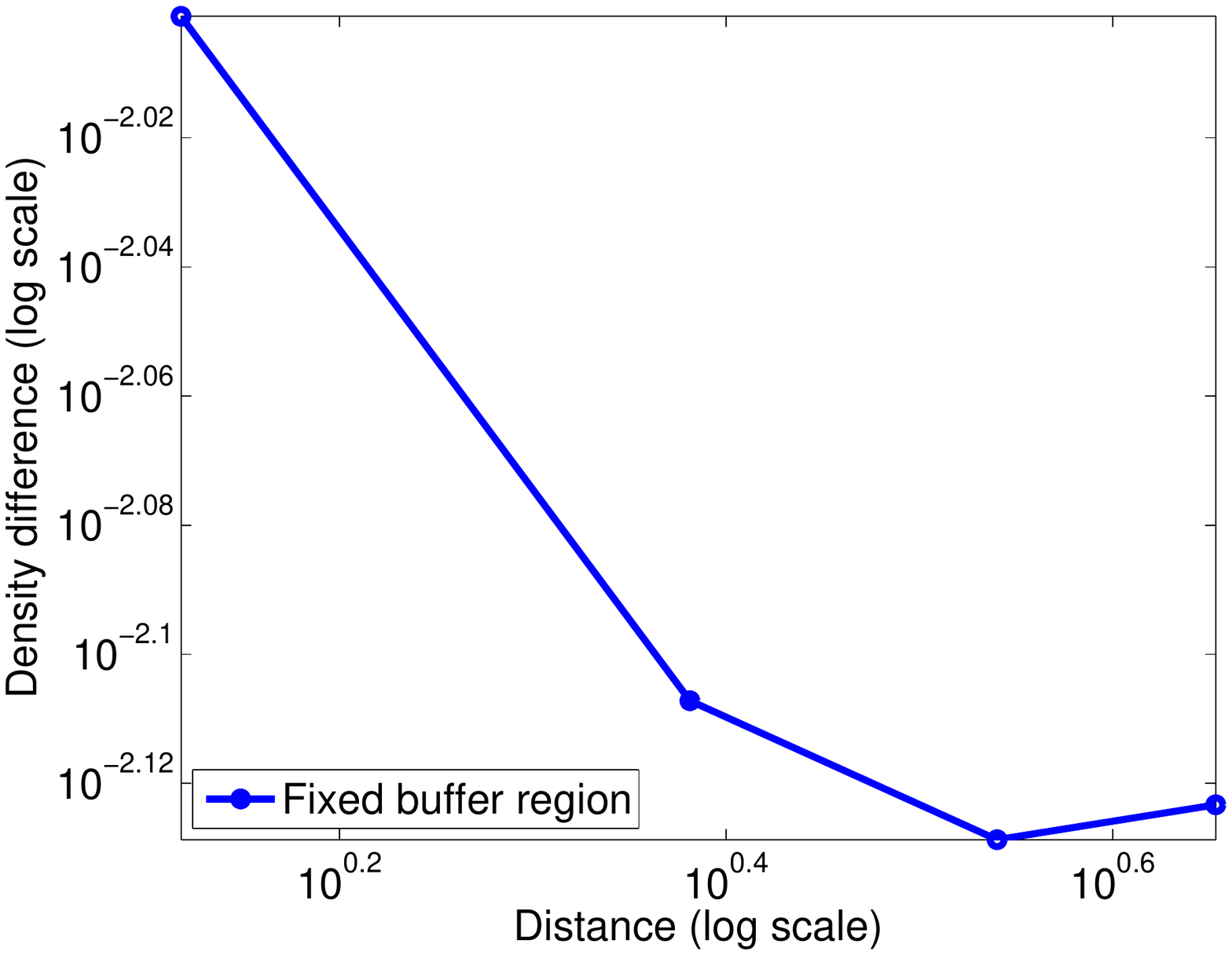}}%
\subfigure[Density difference with PBC]{
\label {fig:metaldensitydiffpbcbuffer}
\includegraphics[width=2.0in]{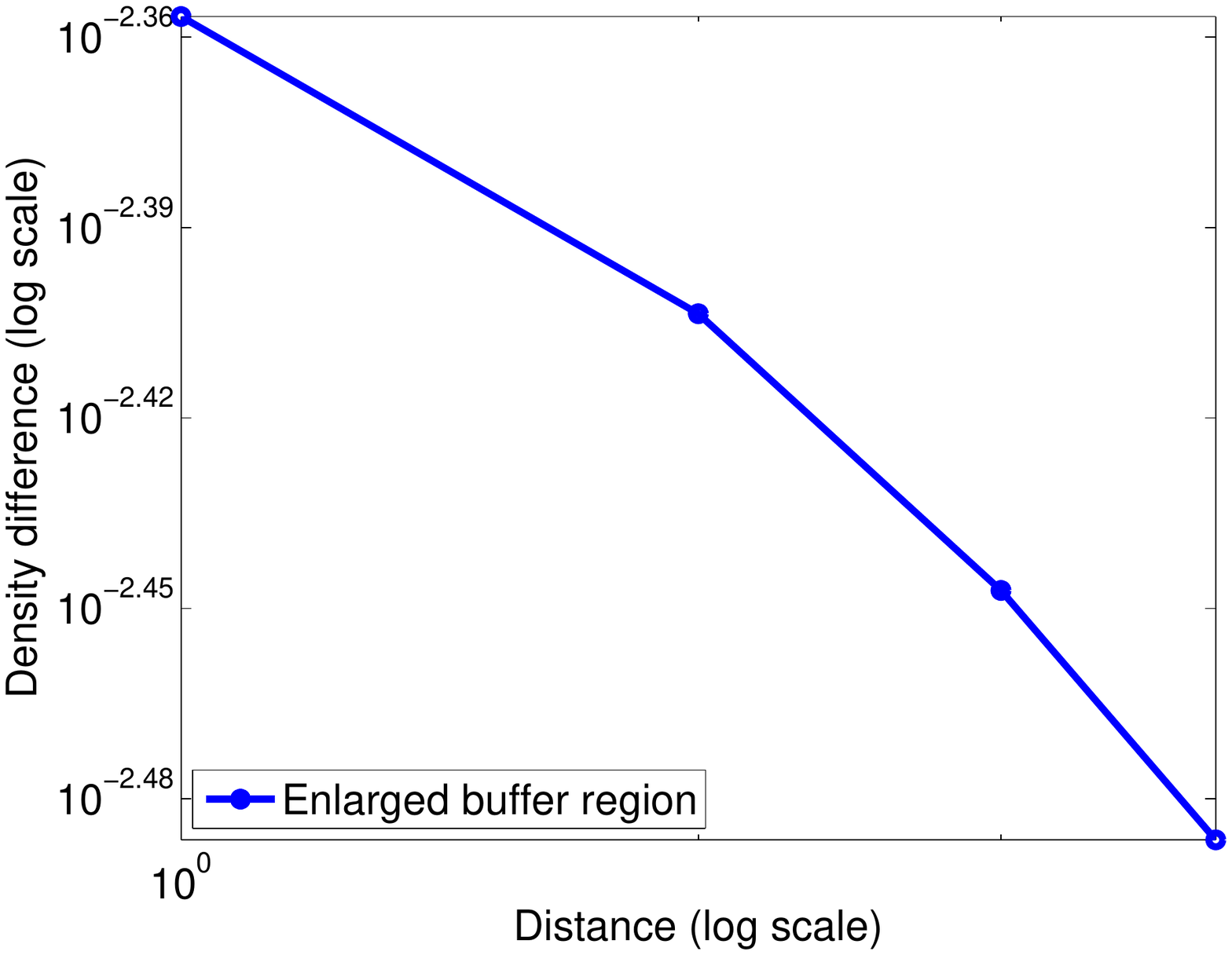}}%
\vspace{-2em}
\caption{\small  $\abs{\rho(0)-\rho_{\Lam}(0)}$ as a function of $x$ in the log-log scale for $\Lam=[0, 16]$ and a series of enlarged buffer regions $\Lam_b=[0-x, 16+x]$ ($x\ge 0.1$) with different boundary conditions in \exref{ex:metal1d}.  (a) Density difference with DBC; (b) Density difference with NBC; (c) Density difference with PBC. Algebraic decay rates are observed in all cases due to the failure of \assumpref{assump}.}\label {fig:metalbuffer}
\end{figure}

\end{example}

\begin{example}[Insulating global system, gap assumption invalid for
  the subsystem]
  \label{ex:insulatormetal} For the next example, consider
\[
V(x)=
\begin{cases}
-a, & 0\leq x\leq 8; \\
-b, & 8 < x < 16;    \\
\text{periodic extension}, & \text{otherwise}.
\end{cases}
\]
We take $a=5$ and $b=0$ and plot the band structure of this problem in
Figure \ref{fig:insulatormetalenergylevelentire}.  It is clear that
this system is an insulating system.
\begin{figure}[htbp]
\vspace{-6em}
\centering
\subfigcapskip -1.0in
\subfigure[Global domain]{%
\label {fig:insulatormetalenergylevelentire}
\includegraphics[width=3.0in]{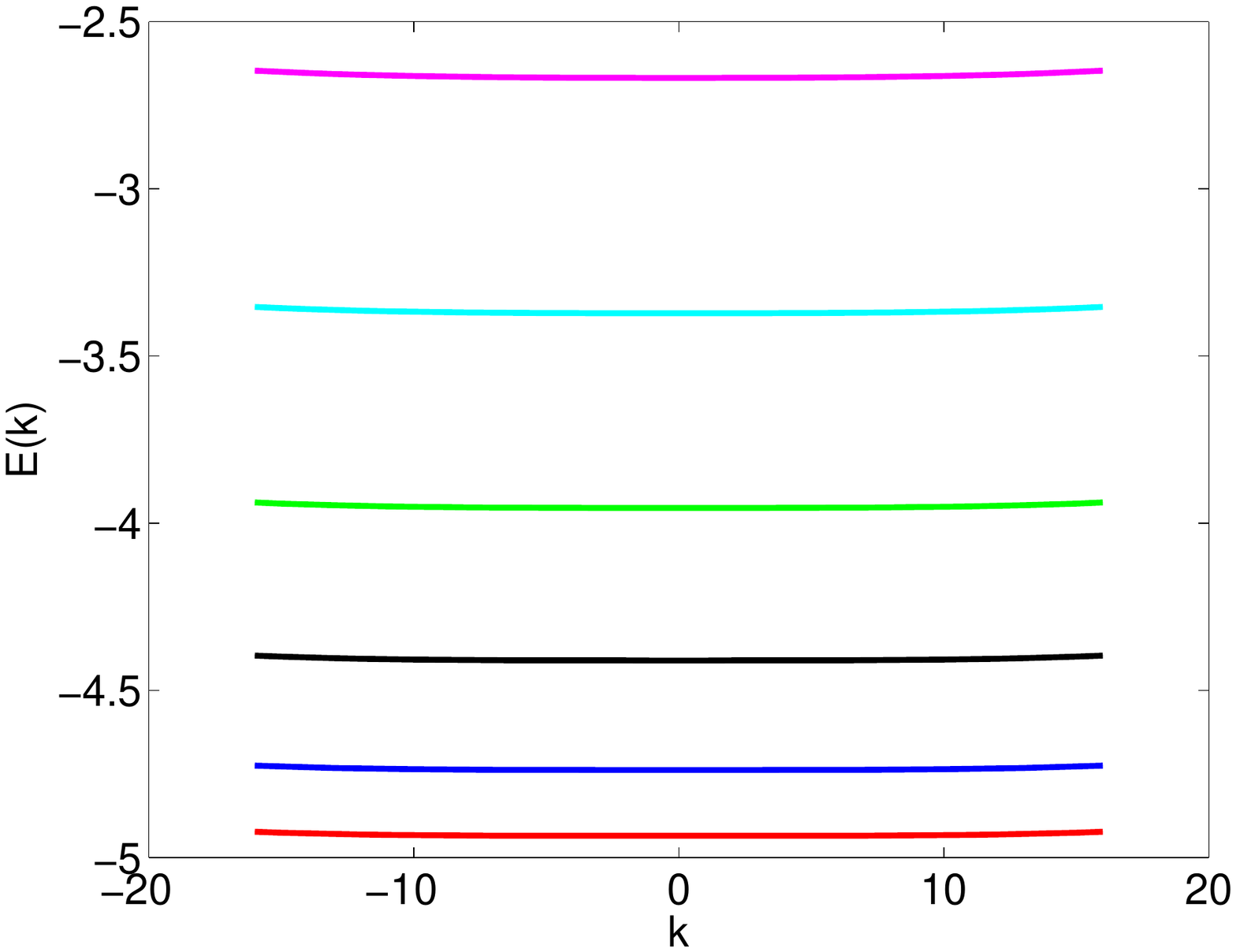}}%
\subfigure[Subdomain]{
\label {fig:insulatormetalenergylevelsub}
\includegraphics[width=3.0in]{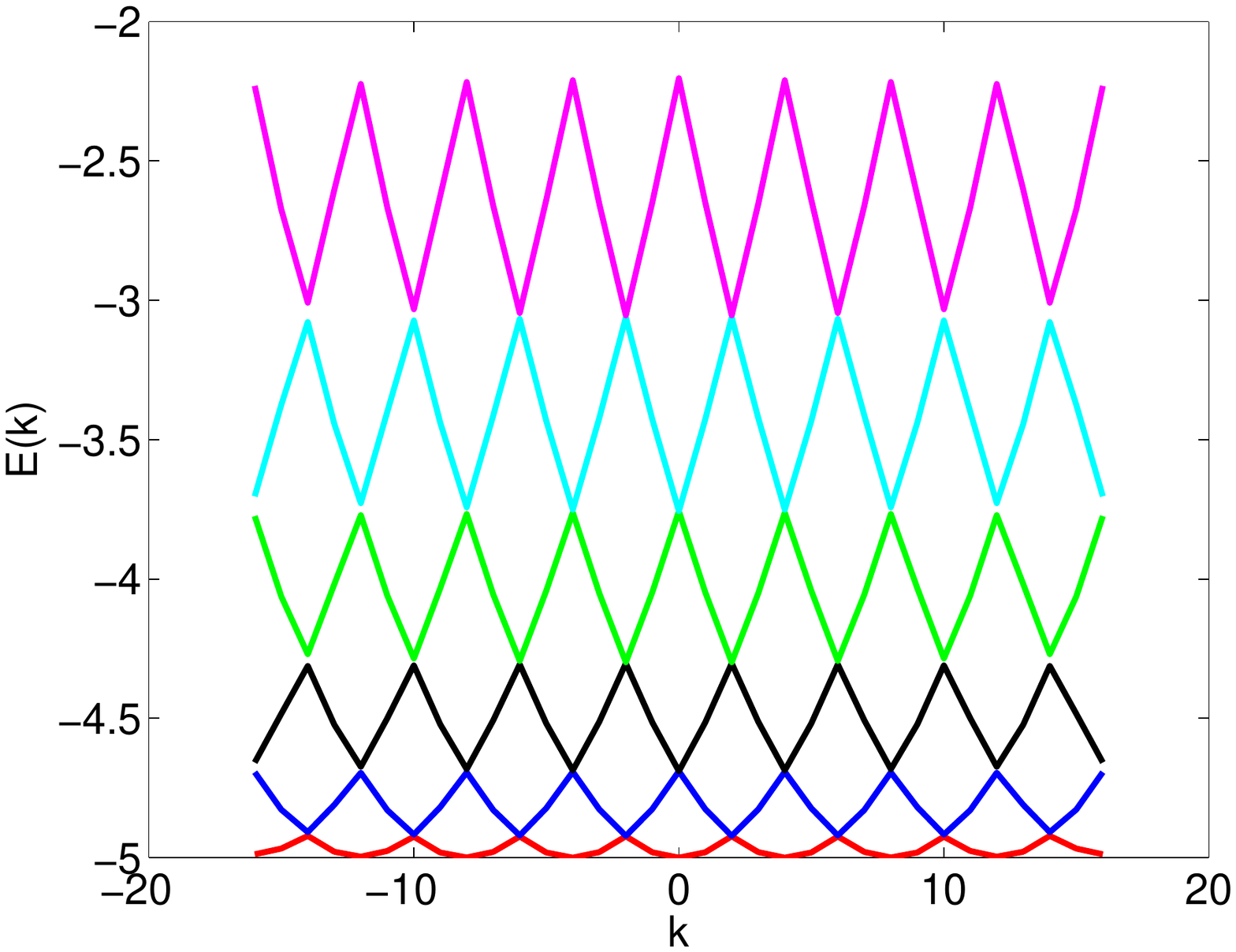}}%
\vspace{-4em}
\caption{\small Band structures of the global system over $[0,16]$ and the subsystem over $[0,8]$ with $a=5$ and $b=0$ in \exref{ex:insulatormetal}. \eqref{assumpa} - \eqref{assumpb} are valid for the global system while
\eqref{assumpc} is invalid for the subsystem. (a) The global system;
(b) The subsystem.}\label {fig:insulatormetalenergylevel}
\end{figure}

First, we fix $\Lam=[0.1,7.9]$ and choose $\Lam_b=[0,8]$. Since
$a=-5$, the subsystem is essentially an eigenvalue problem of the
Laplacian operator, which means \eqref{assumpc} is not valid; see
Figure~\ref{fig:insulatormetalenergylevelsub}. Figure
\ref{fig:insulatormetaldensitydifference1} shows that only algebraic
decay rate is observed, since \eqref{assumpa}--\eqref{assumpb} are
valid while \eqref{assumpc} is invalid. Second, for the fixed
$\Lam=[3,5]$, we choose a series of enlarged buffer regions $\Lam_b =
[3-x,5+x]$ ($x\ge0.1$) by varying $x$. $\abs{\rho(5)-\rho_{\Lam}(5)}$
as a function of $x$ is shown in
Figure~\ref{fig:insulatormetaldensitydifference2}.  Algebraic decay
rate is also observed since \eqref{assumpc} is invalid in this
case. Finally, we fix $\Lam=[0,8]$ and choose a series of enlarged
buffer regions $\Lam_b = [0-x,8+x]$ ($x\ge0.1$) by varying
$x$. $\abs{\rho(8)-\rho_{\Lam}(8)}$ as a function of $x$ is shown in
Figure \ref{fig:insulatormetaldensitydifference3}.  Exponential decay
rate is observed since \eqref{assumpc} becomes valid in this case.

\begin{figure}[htbp]
\vspace{-4em}
\centering
\subfigcapskip -0.5in
\subfigure[Fixed buffer region]{%
\label {fig:insulatormetaldensitydifference1}
\includegraphics[width=2.0in]{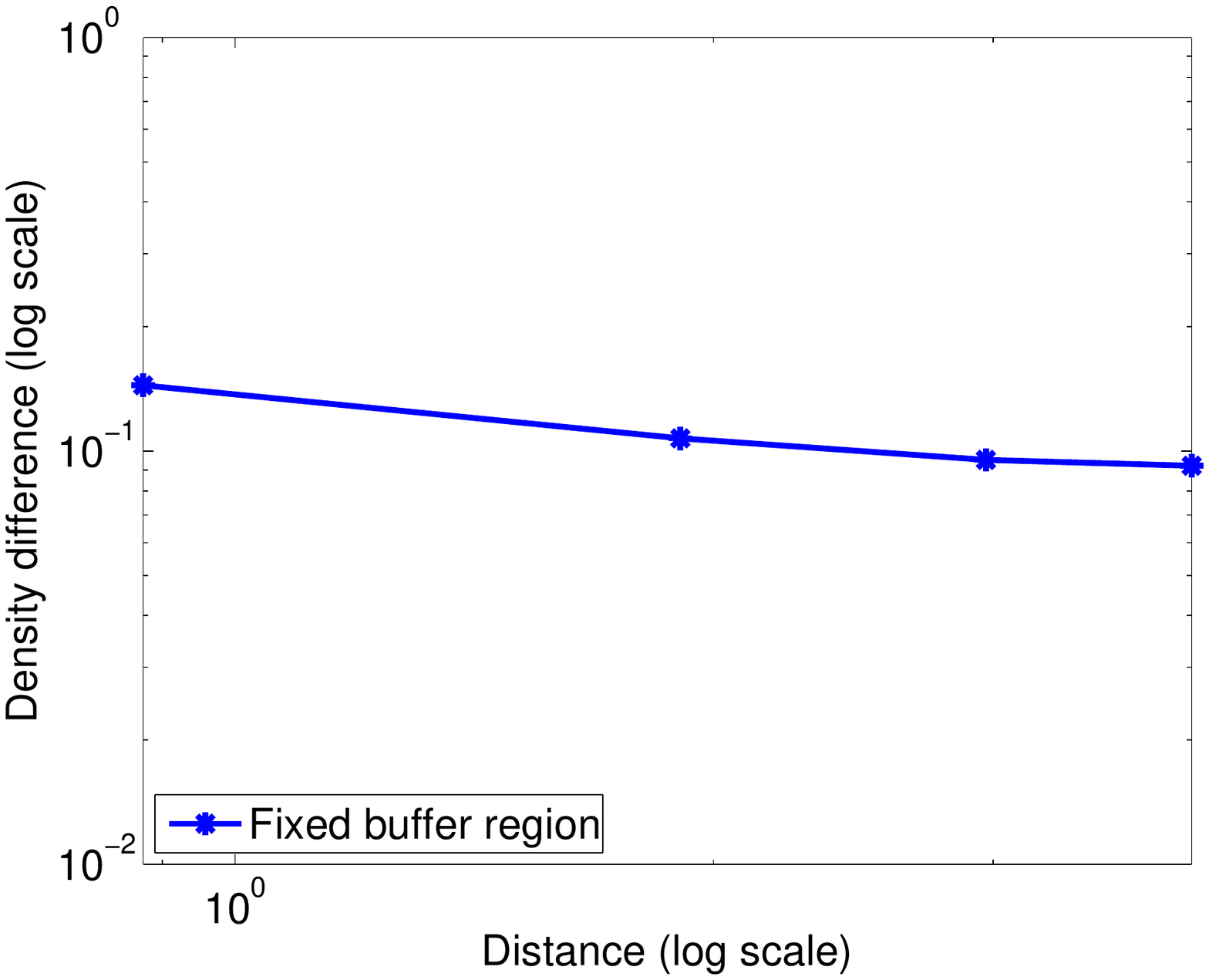}}%
\subfigure[Enlarged buffer region]{
\label {fig:insulatormetaldensitydifference2}
\includegraphics[width=2.0in]{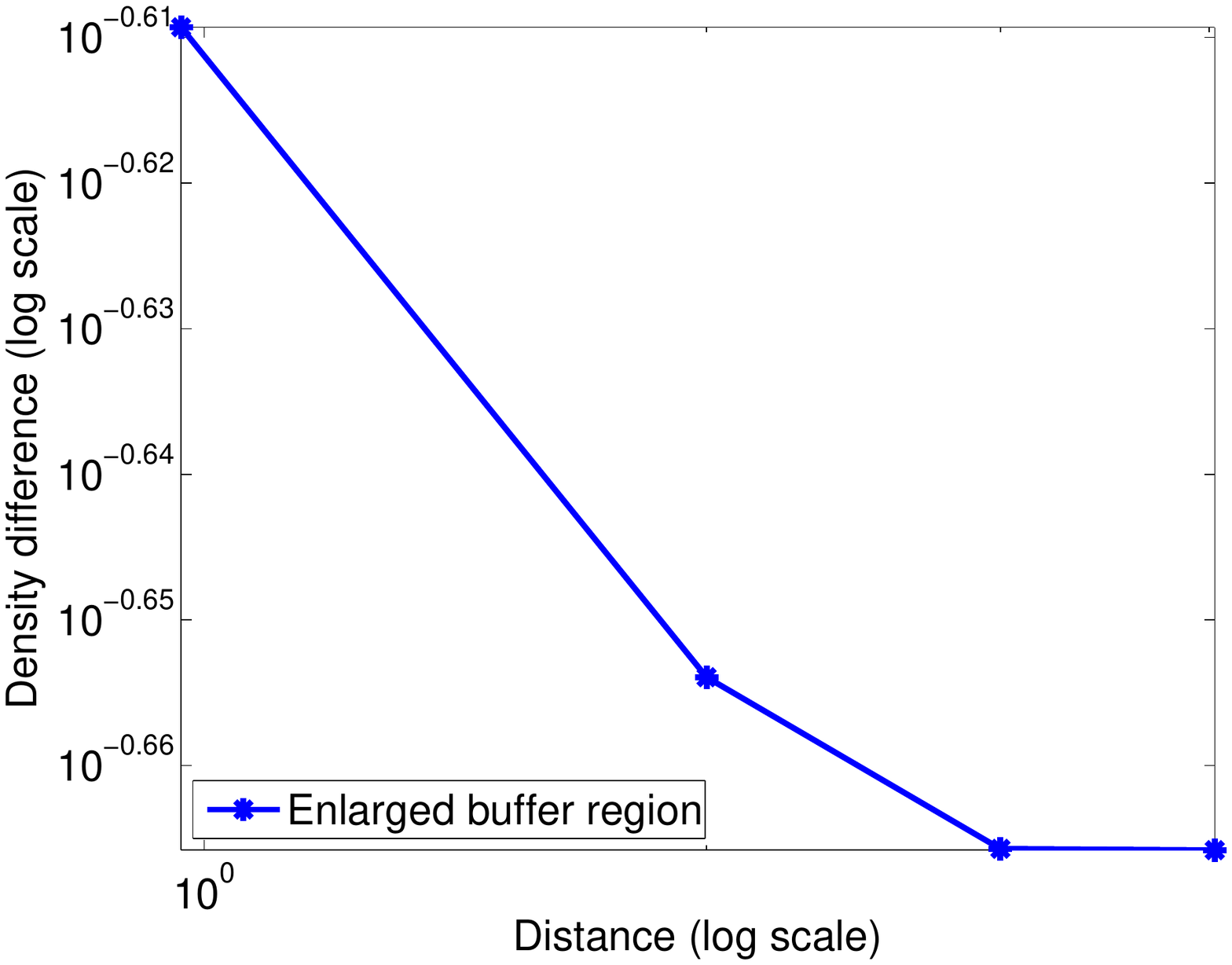}}%
\subfigure[Enlarged buffer region]{
\label {fig:insulatormetaldensitydifference3}
\includegraphics[width=2.0in]{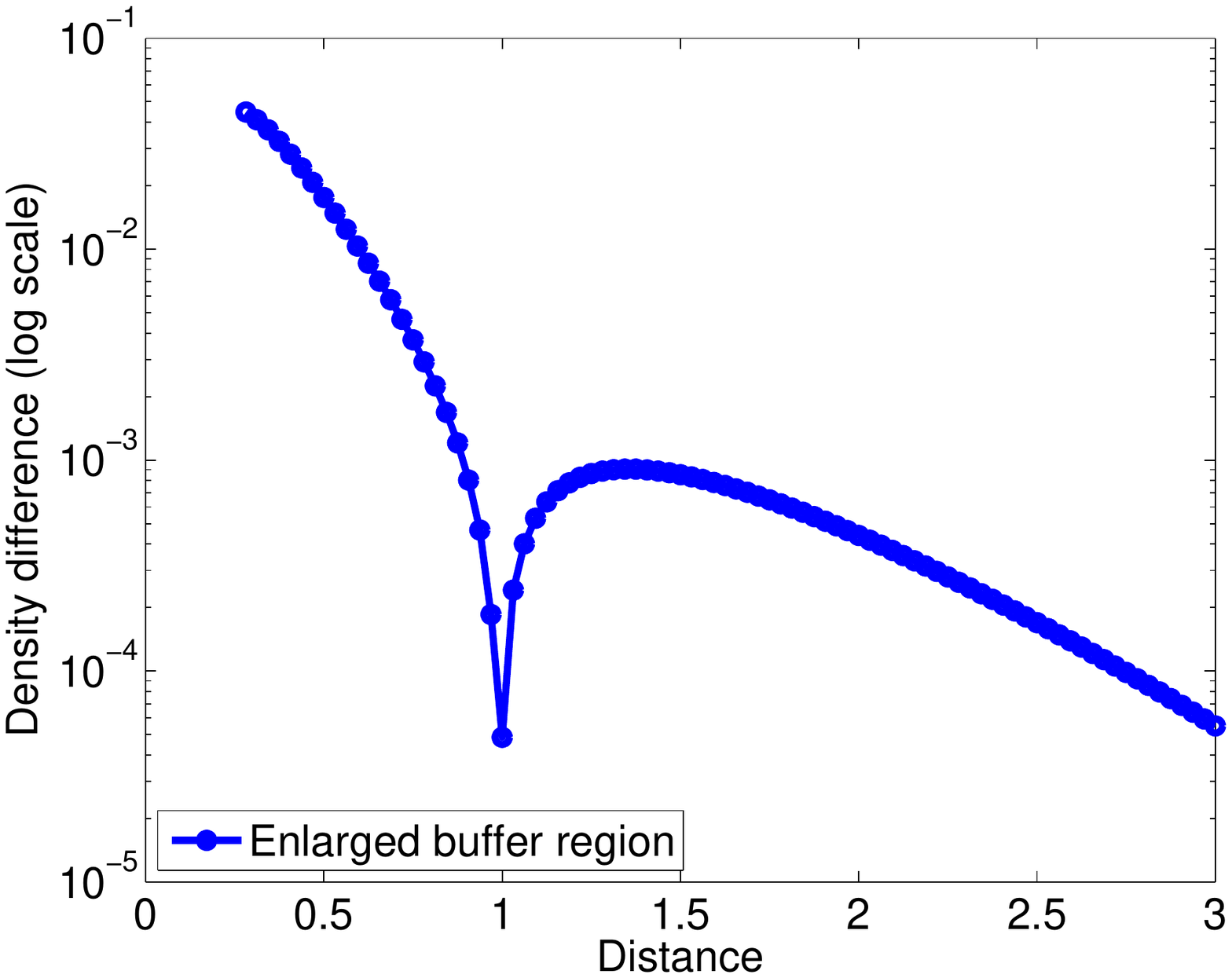}}%
\vspace{-2em}
\caption{\small Difference between electron densities of the global system and subsystems
as a function of distance in \exref{ex:insulatormetal}. (a) $\abs{\rho(x)-\rho_{\Lam}(x)}$
with $\Lam=[0.1,7.9]$ and $\Lam_{b}=[0,8]$ as a function of $x$;
(b) $\abs{\rho(5)-\rho_{\Lam}(5)}$
with $\Lam=[3,5]$ and $\Lam_{b}=[3-x,5+x]$ ($x\ge0.1$) as a function of $x$; (c) $\abs{\rho(8)-\rho_{\Lam}(8)}$
with $\Lam=[0,8]$ and $\Lam_{b}=[0-x,8+x]$ ($x\ge0.1$) as a function of $x$. Exponential decay rate is observed in (c), while only algebraic decay rates are observed in (a) and (b), due to the validity and invalidity of \eqref{assumpc}
in corresponding cases.
}\label {fig:insulatormetaldensitydifference}
\end{figure}
\end{example}

\begin{example}
\label{ex:2d}
Consider an infinite array of atoms on a two-dimensional lattice with $X_i=i, Y_j=j$, for $i\in\mathbb{Z}, j\in\mathbb{Z}$. Each atom has one valence electron and spin degeneracy is ignored. $V$ is of the following form
\[
V(x,y)= -\sum_{i\in\mathbb{Z},j\in\mathbb{Z}}\frac{a}{\sqrt{2\pi\sigma^2}}\exp{[-(x-X_i)^2/2\sigma^2-(y-Y_j)^2/2\sigma^2]}.
\]

Figure \ref{fig:bandstructure2d} shows band structures when $a=10$, $\sigma=0.15$, and $a=10$, $\sigma=0.45$.
It is clear that for selected parameters, the corresponding system is an insulator in Figure \ref{fig:bandstructure2dinsulator}, while it is a metal in Figure \ref{fig:bandstructure2dmetal}.
\begin{figure}[htbp]
\vspace{-6em}
\centering
\subfigcapskip -0.9in
\subfigure[Insulator]{%
\label {fig:bandstructure2dinsulator}
\includegraphics[width=3.0in]{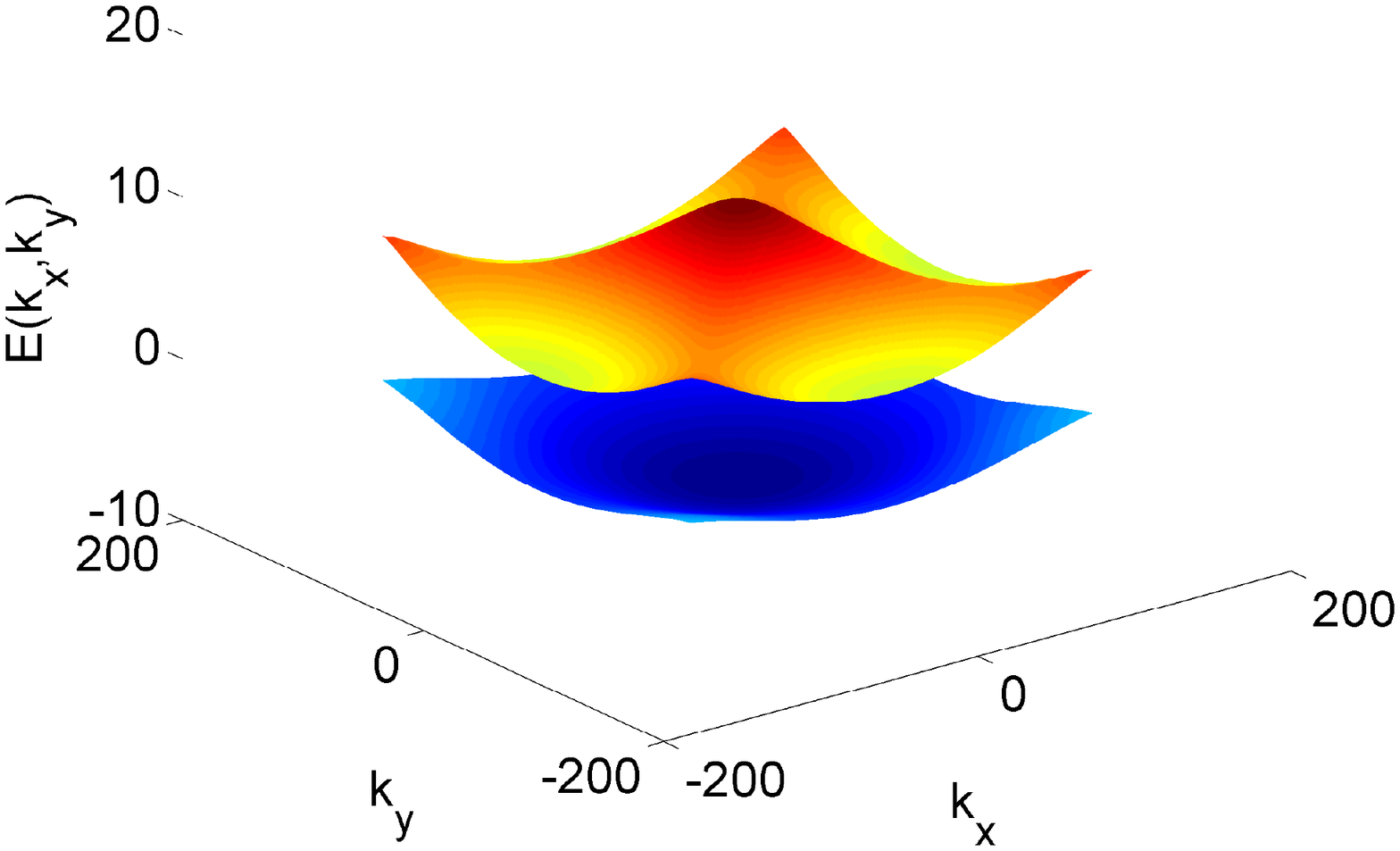}}%
\subfigure[Metal]{
\label {fig:bandstructure2dmetal}
\includegraphics[width=3.0in]{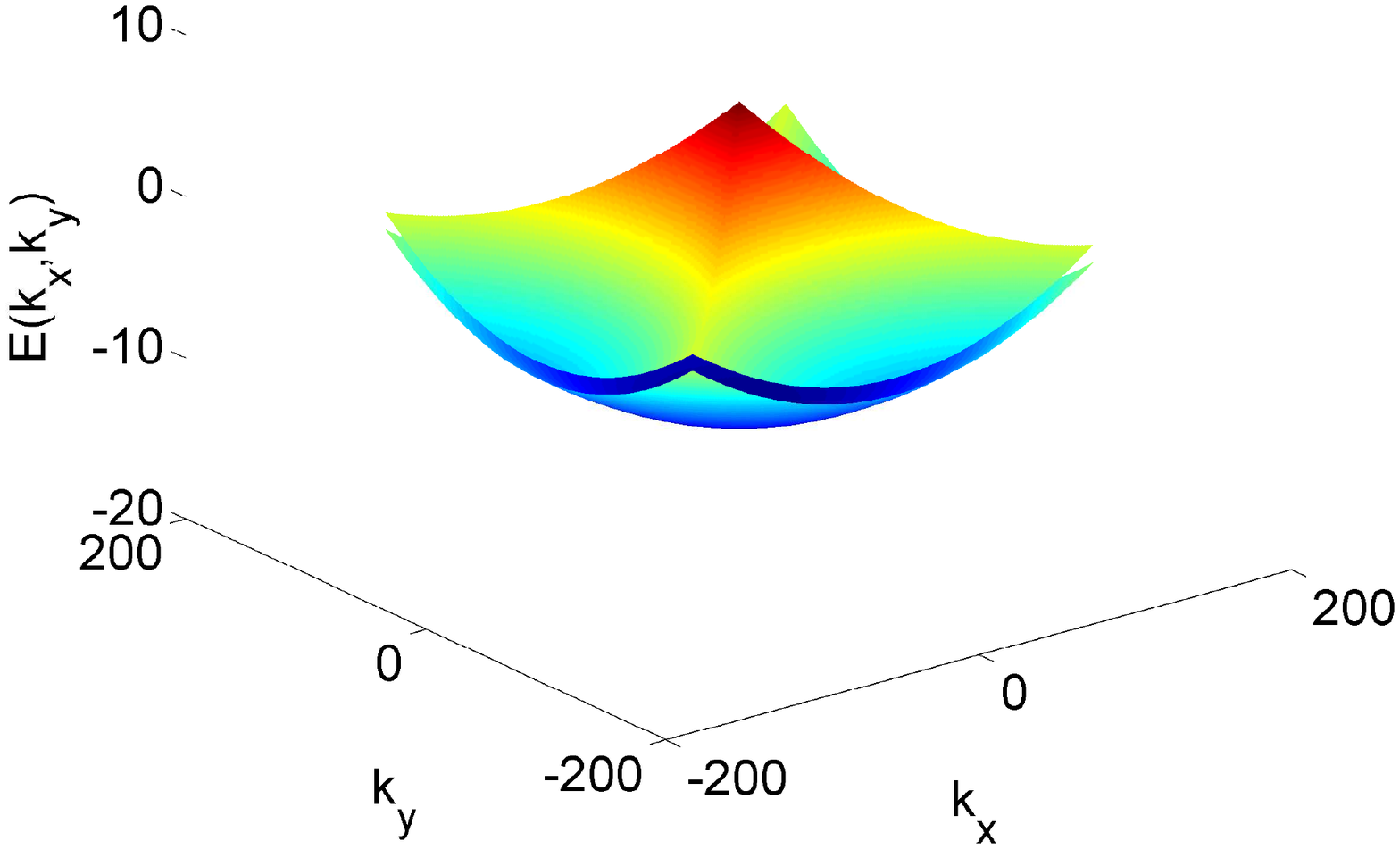}}%
\vspace{-4em}
\caption{\small  Band structures for different parameters in \exref{ex:2d}. (a) Insulator, where $a=10$ and $\sigma=0.15$; (b) Metal, where $a=10$ and $\sigma=0.45$.}\label {fig:bandstructure2d}
\end{figure}

First, we fix $\Lam = [0.1,5.9]\times[0.1,5.9]$ and choose $\Lam_b = [0,
6]\times[0, 6]$. All three boundary conditions are
tested. $\abs{\rho(x)-\rho_{\Lam}(x)}$ is plotted in the centered row
of Figure \ref{fig:dendiff2d} for DBC and PBC. Second, for the fixed
$\Lam = [2,4]\times[2,4]$, we consider a series of enlarged buffer
regions $\Lam_b = [2-x,4+x]\times[2-x,4+x]$
($x\ge0.1$). $\abs{\rho(2,2)-\rho_{\Lam}(2,2)}$ is plotted in the
bottom row of Figure \ref{fig:dendiff2d}. The left column of Figure
\ref{fig:dendiff2d} is for the insulator case, while the right column
is for the metal case. Results here are consistent with theoretical
estimates.
\begin{figure}[htbp]
\vspace{-6em}
\centering
\subfigcapskip -0.8in
\subfigure[Energy level with DBC]{%
\label {fig:insulatorenergyleveldbc2d}
\includegraphics[width=2.5in]{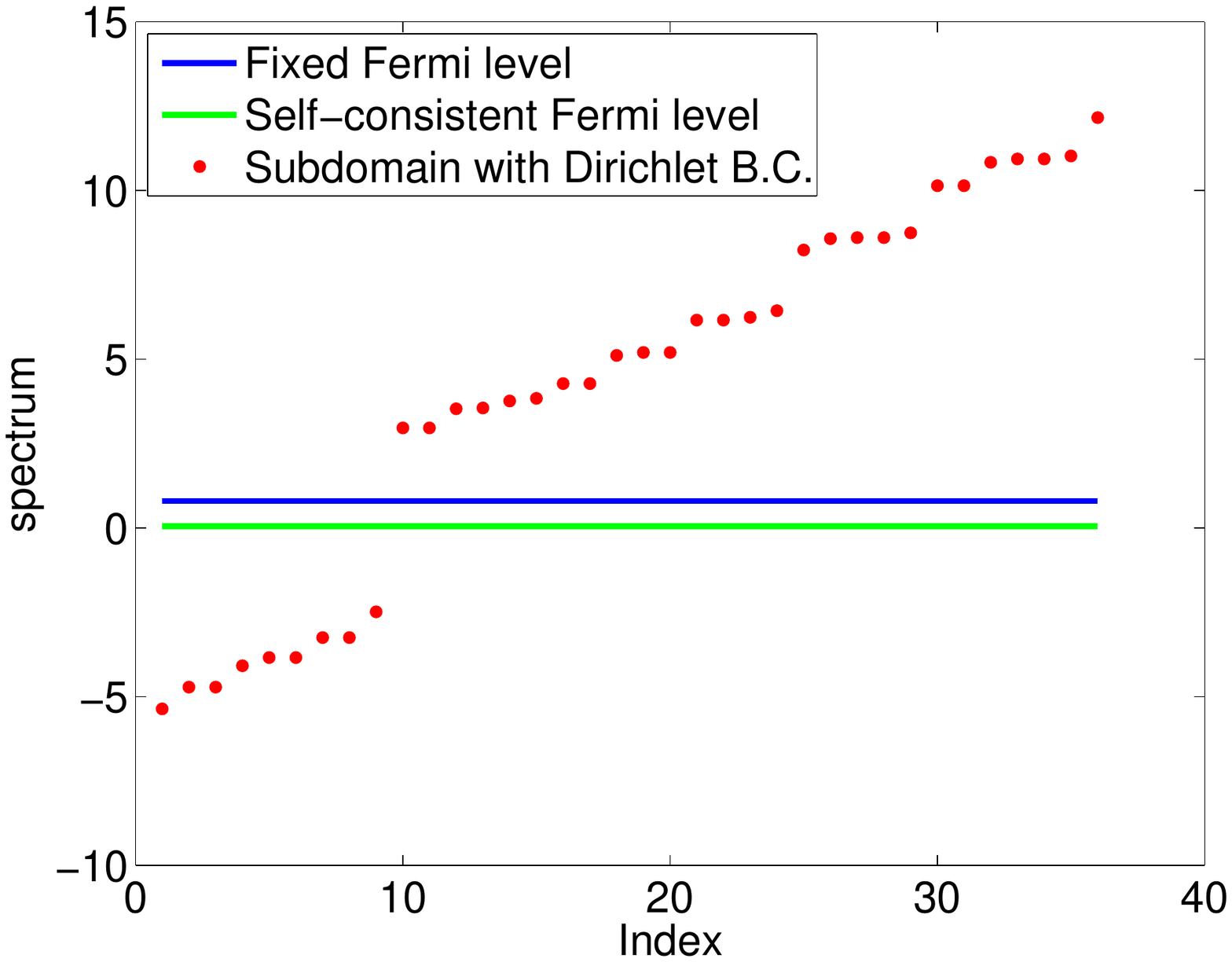}}%
\subfigure[Energy level with PBC]{%
\label {fig:metalenergylevelpbc2d}
\includegraphics[width=2.5in]{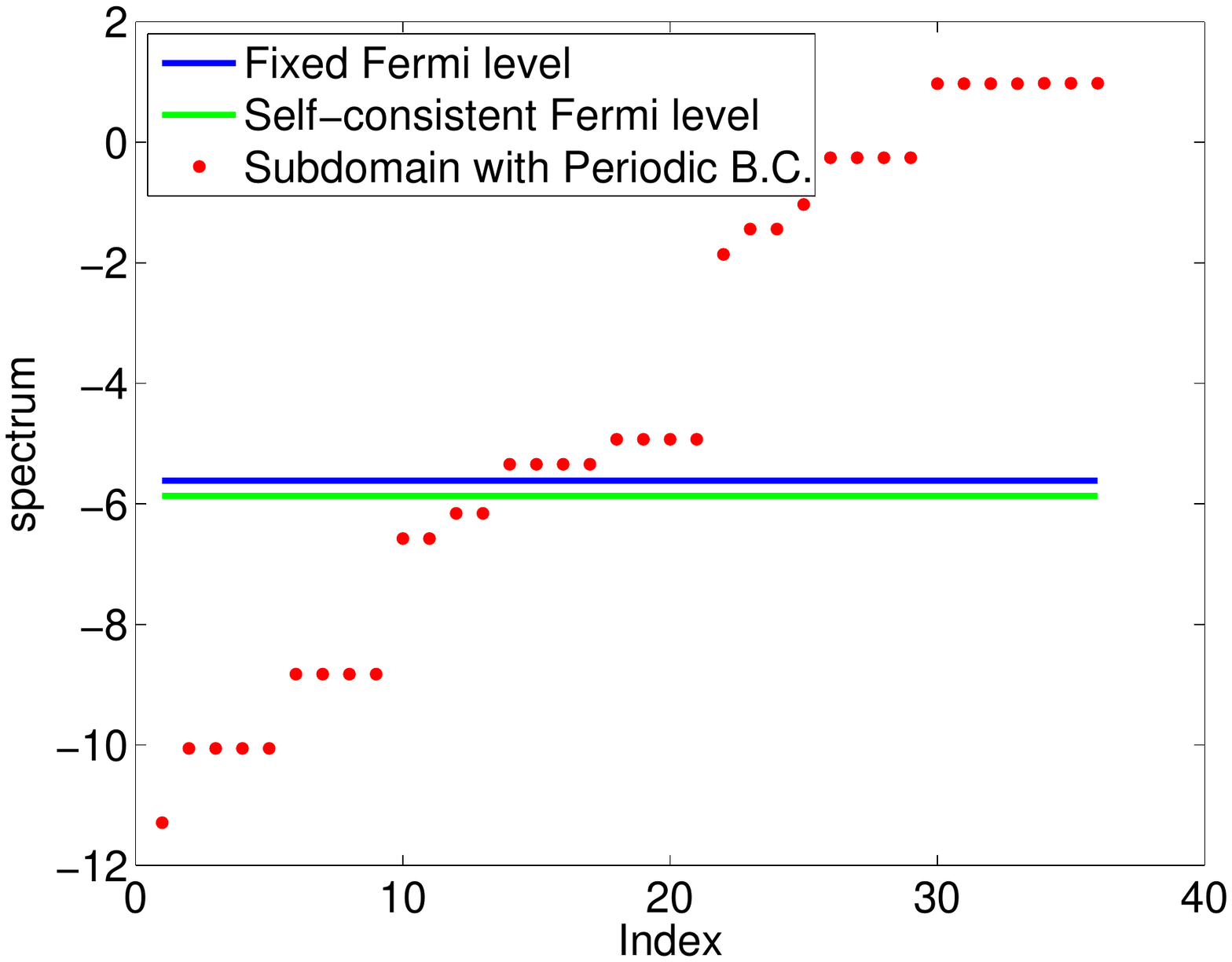}}%
\\
\vspace{-8em}
\subfigure[Density difference with DBC]{
\label {fig:insulatordensitydiffdbc2d}
\includegraphics[width=2.5in]{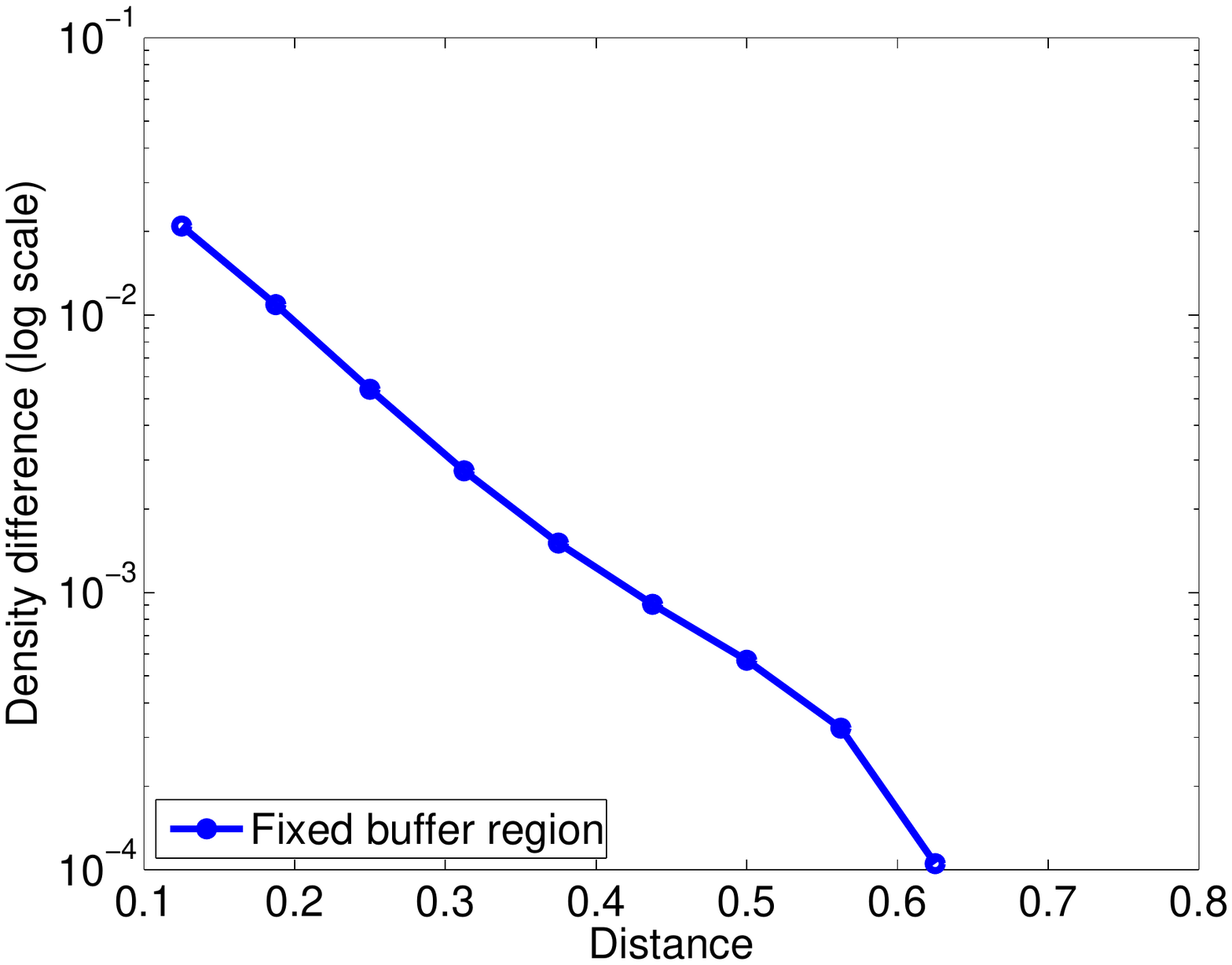}}%
\subfigure[Density difference with PBC]{
\label {fig:metaldensitydiffpbc2d}
\includegraphics[width=2.5in]{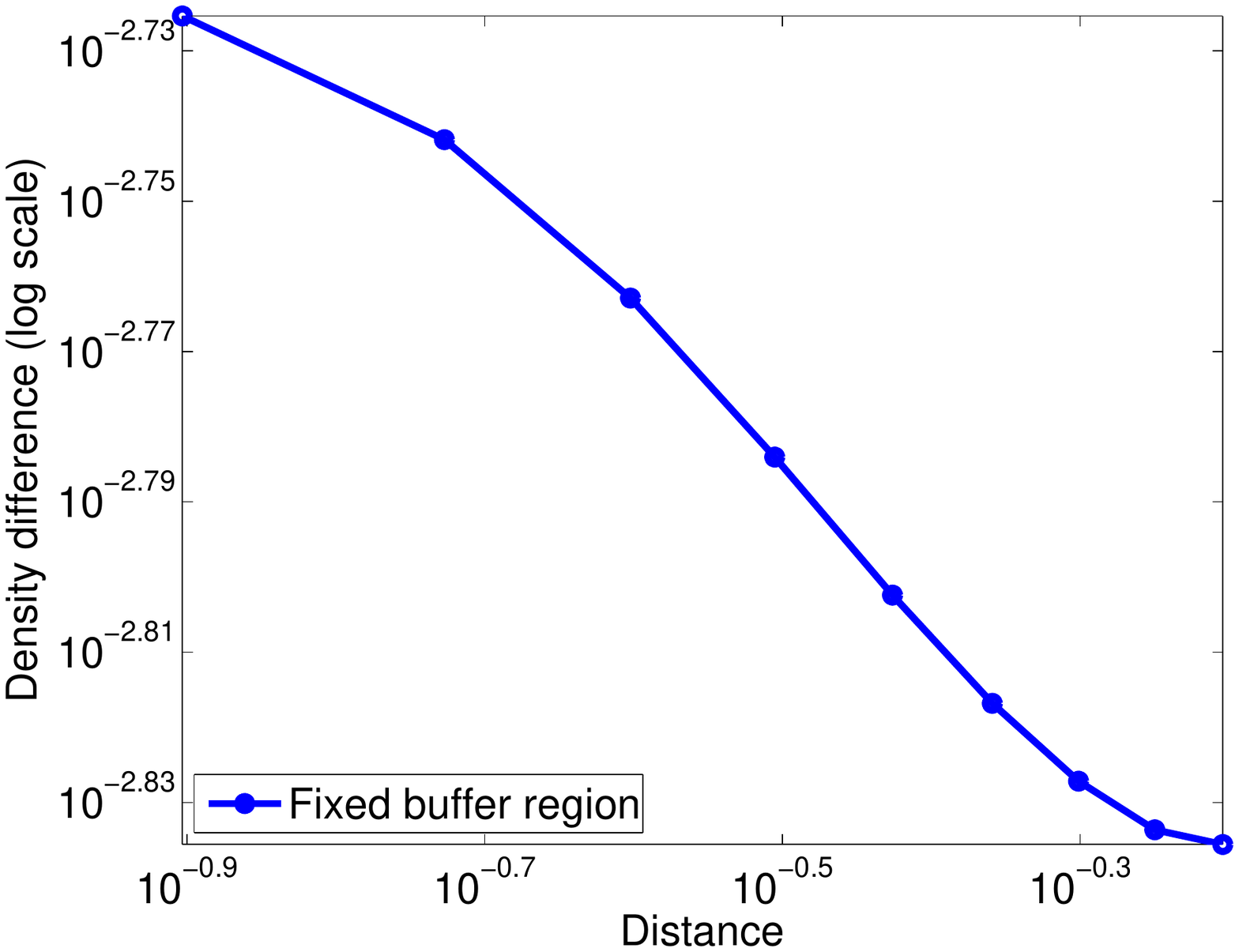}}%
\\
\vspace{-8em}
\subfigure[Density difference with DBC]{
\label {fig:insulatordensitydiffdbcbuffer2d}
\includegraphics[width=2.5in]{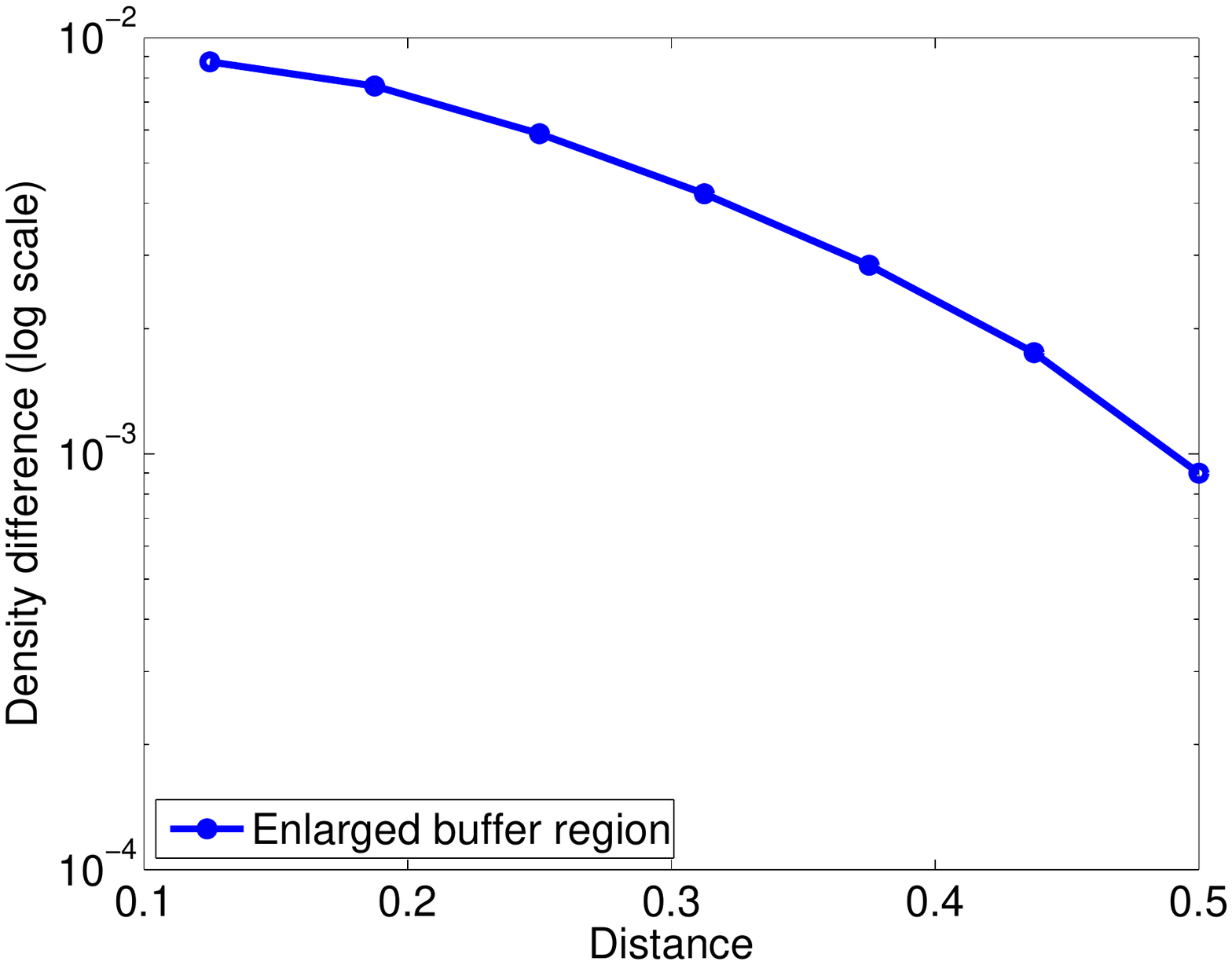}}%
\subfigure[Density difference with PBC]{
\label {fig:metaldensitydiffpbcbuffer2d}
\includegraphics[width=2.5in]{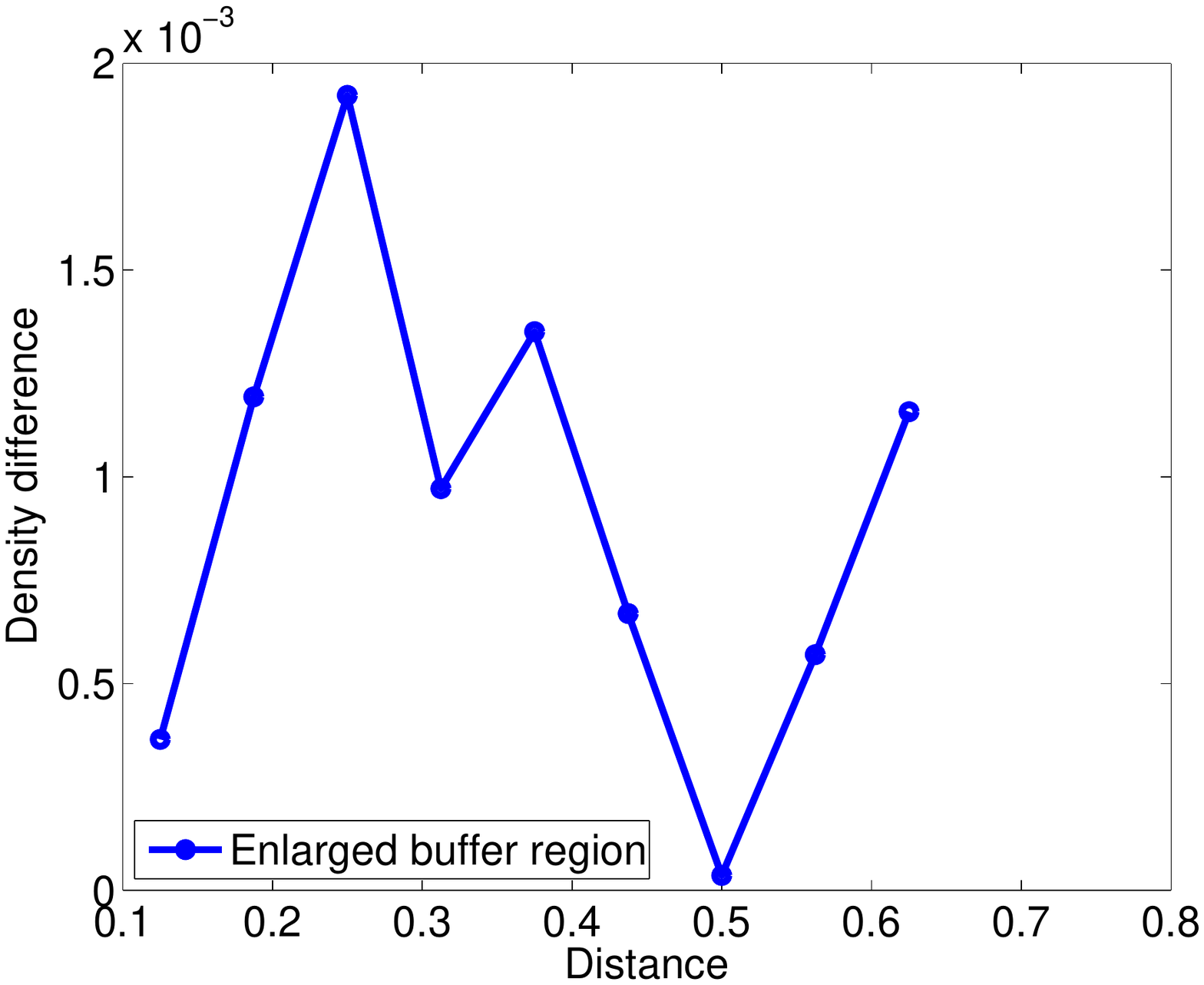}}%
\vspace{-4em}
\caption{\small  Energy levels of the subsystem (top row), $\abs{\rho(x)-\rho_{\Lam}(x)}$ with $\Lam=[0.1,5.9]\times[0.1,5.9]$ and $\Lam_b=[0,6]\times[0,6]$ as a function of $x$ (centered row),
and $\abs{\rho(2,2)-\rho_{\Lam}(2,2)}$ with $\Lam=[2,4]\times[2,4]$ and $\Lam_b = [2-x,4+x]\times[2-x,4+x]$ ($x\ge0.1$) (bottom row) in \exref{ex:2d}.
Left column: Insulator; Right column: Metal. (a) Energy level with DBC; (b) Energy level with PBC; (c) Density difference with DBC; (b) Density difference with PBC; (e) Density difference with DBC; (f) Density difference with PBC. Decay rates are consistent
with theoretical estimates.}\label {fig:dendiff2d}
\end{figure}
\end{example}

\begin{example}[Insulating global system, gap assumption invalid for
  the subsystem]
\label{ex:insulatormetal2d}
Consider
\[
V(x,y)=
\begin{cases}
-a, & (x,y)\in \{ 3< x< 9, 3< y< 9\}; \\
-b, & (x,y)\in \{ [0, 12]\times[0,12]\}/\{3 < x < 9, 3< y< 9\}; \\
\text{periodic extension}, & \text{otherwise}.
\end{cases}
\]
Choose $a=5$ and $b=0$ and plot the band structure of this problem in
Figure \ref{fig:insulatormetalenergylevelentire2d}. It is clear that
\eqref{assumpa}--\eqref{assumpb} are valid.
\begin{figure}[htbp]
\vspace{-6em}
\centering
\subfigcapskip -1.0in
\subfigure[Global domain]{%
\label {fig:insulatormetalenergylevelentire2d}
\includegraphics[width=3.0in]{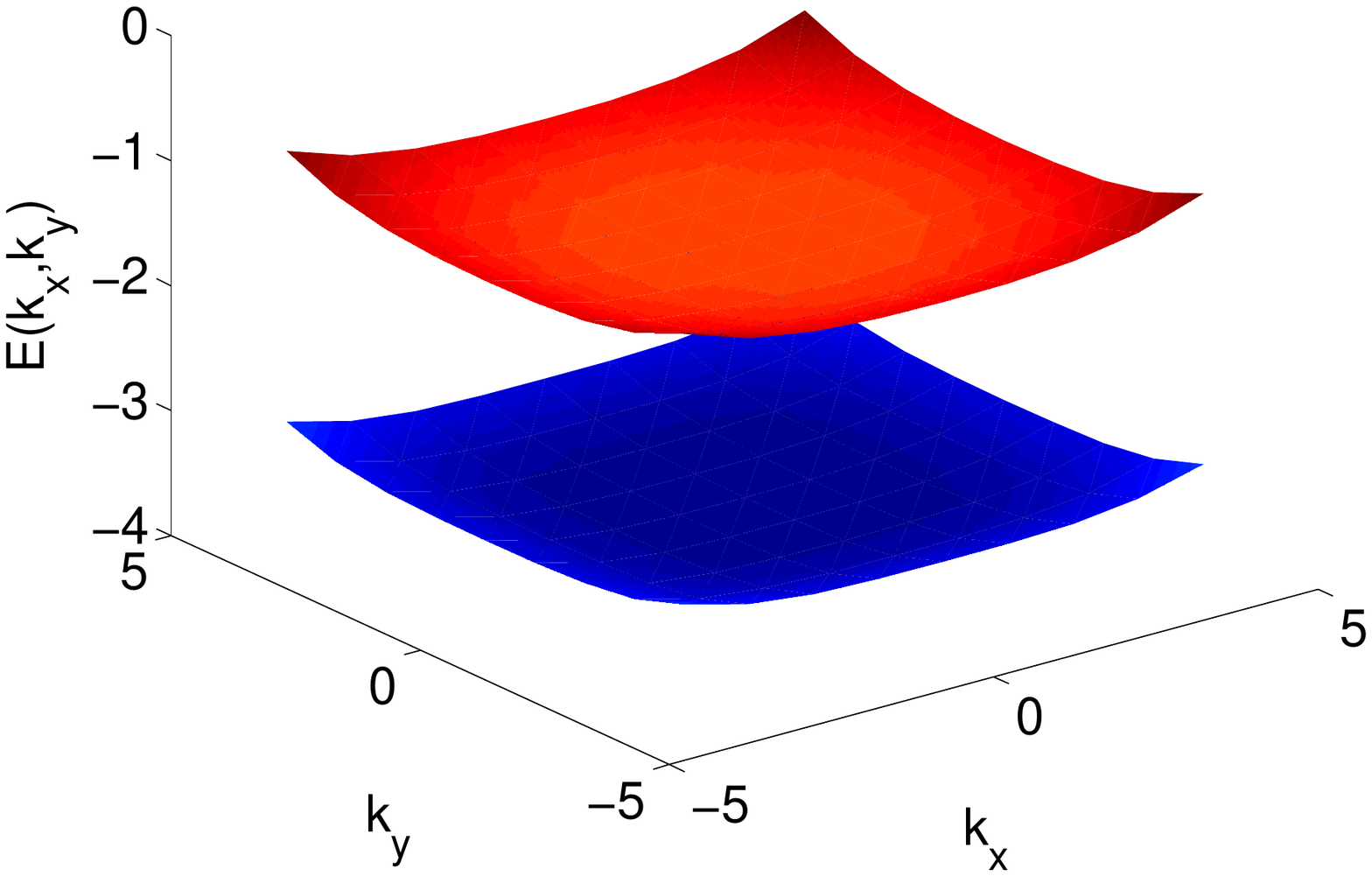}}%
\subfigure[Subdomain]{
\label {fig:insulatormetalenergylevelsub2d}
\includegraphics[width=3.0in]{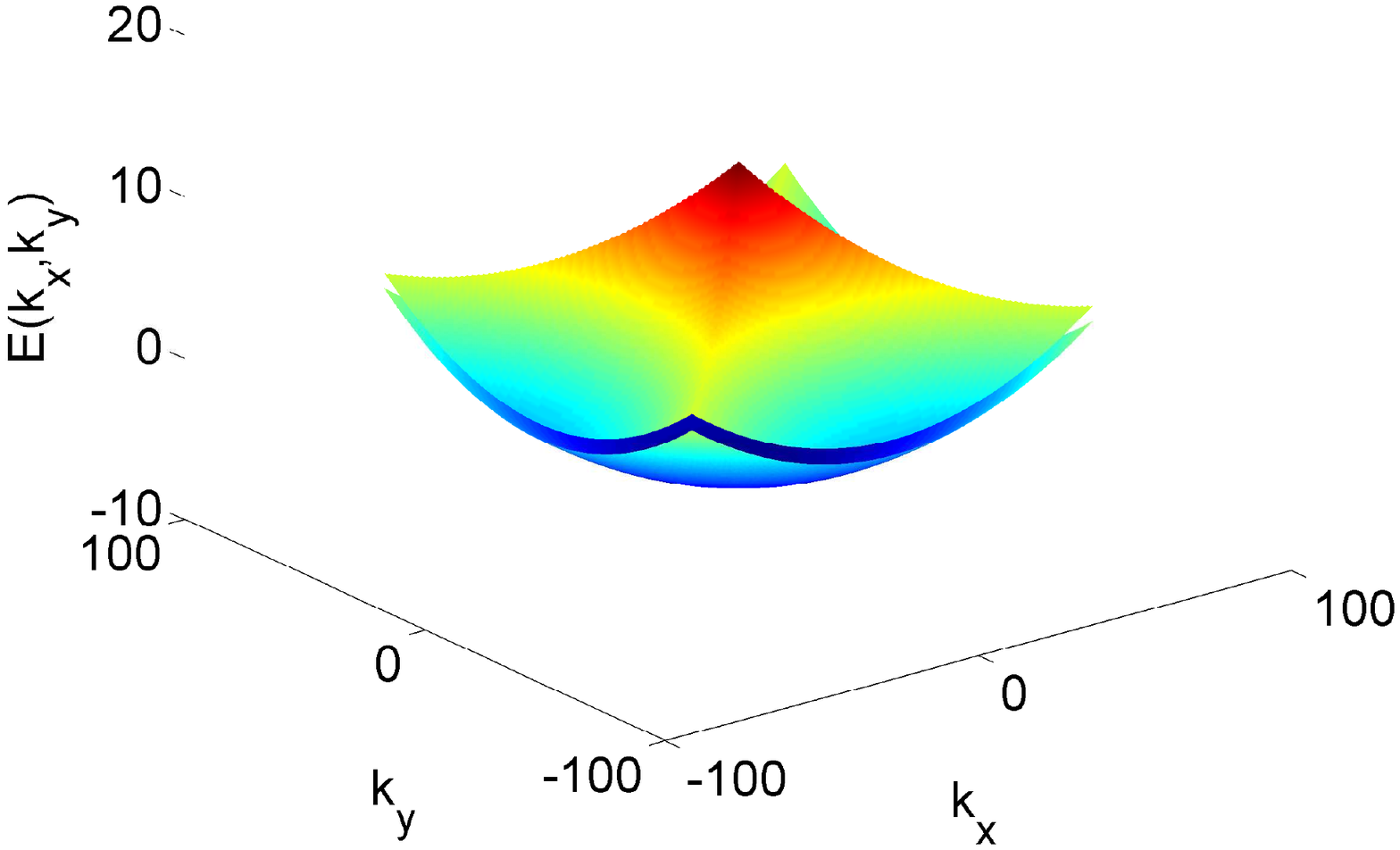}}%
\vspace{-4em}
\caption{\small Band structures of the global system over $[0,12]\times[0,12]$ and the subsystem over $[3,9]\times[3,9]$ when $a=5$ and $b=0$ in \exref{ex:insulatormetal2d}.
(a) The global system; (b) The subsystem.}\label {fig:insulatormetalenergylevel2d}
\end{figure}

Now we fix $\Lam=[3.1,8.9]\times[3.1,8.9]$ and choose $\Lam_b=[3,9]\times[3,9]$. Since $a=-5$, the subsystem is essentially the eigenvalue problem of the Laplacian operator, which implies the invalidity of \eqref{assumpc}.
As a consequence, only algebraic decay rate is
observed in Figure \ref{fig:insulatormetaldensitydifference2d1}. Furthermore, we fix $\Lam=[5, 7]\times[5,7]$
and choose a series of enlarged buffer regions $\Lam_b=[5-x,7+x]\times[5-x,7+x]$ ($x\ge0.1$) by varying $x$.
$\abs{\rho(5,5)-\rho_{\Lam}(5,5)}$ as a function of $x$ is shown in Figure \ref{fig:insulatormetaldensitydifference2d2}. Algebraic decay rate is also observed again
due to the invalidity of \eqref{assumpc}. Finally, we fix $\Lam=[3,9]\times[3,9]$ and choose a series of enlarged buffer regions $\Lam_b=[3-x,9+x]\times[3-x,9+x]$ ($x\ge0.1$). $\abs{\rho(3,3)-\rho_{\Lam}(3,3)}$
as a function of $x$ is shown in Figure \ref{fig:insulatormetaldensitydifference2d3}.
Exponential decay rate is observed since \eqref{assumpc} becomes valid in this case.

\begin{figure}[htbp]
\vspace{-4em}
\centering
\subfigcapskip -0.5in
\subfigure[Fixed buffer region]{%
\label {fig:insulatormetaldensitydifference2d1}
\includegraphics[width=2.0in]{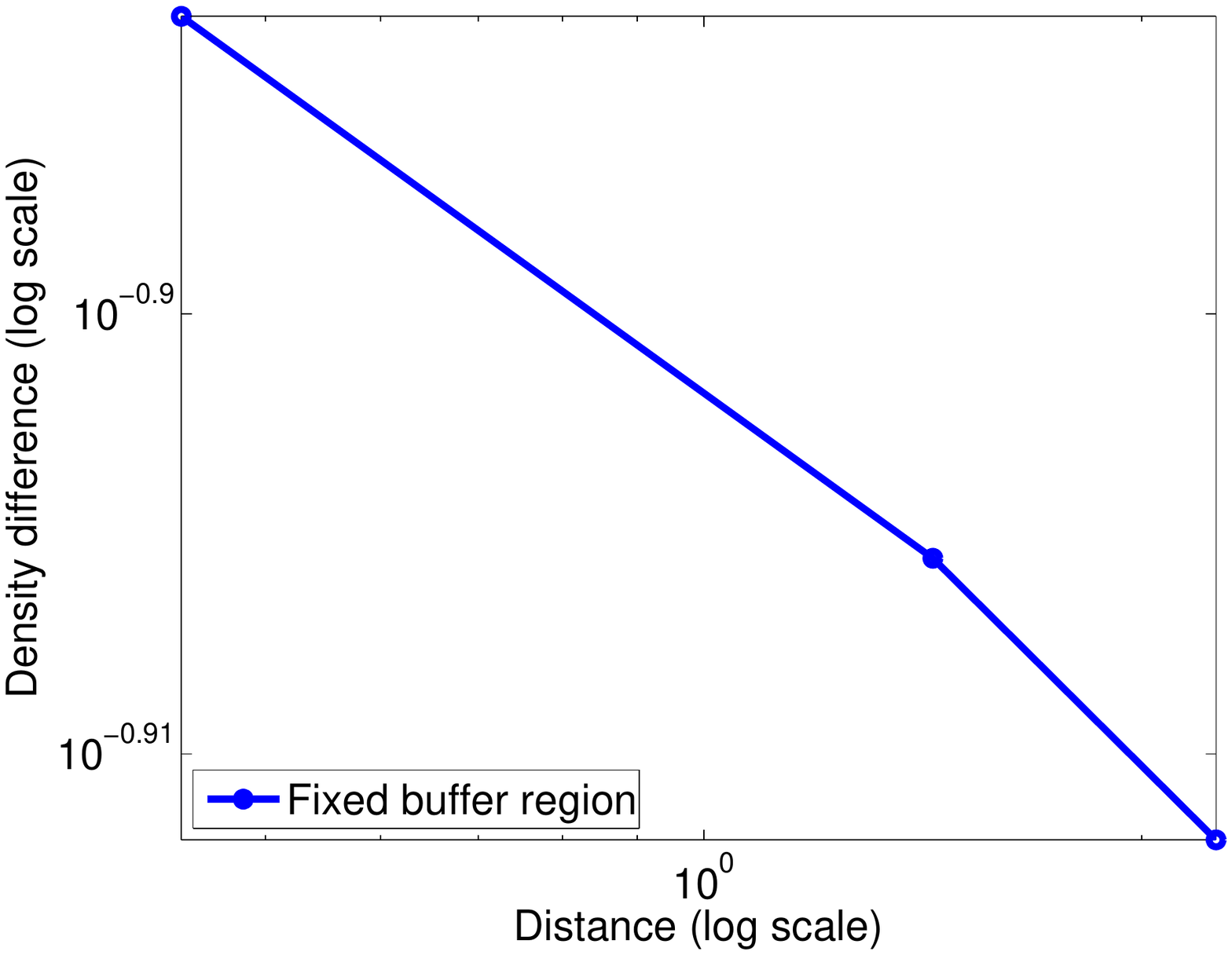}}%
\subfigure[Enlarged buffer region]{
\label {fig:insulatormetaldensitydifference2d2}
\includegraphics[width=2.0in]{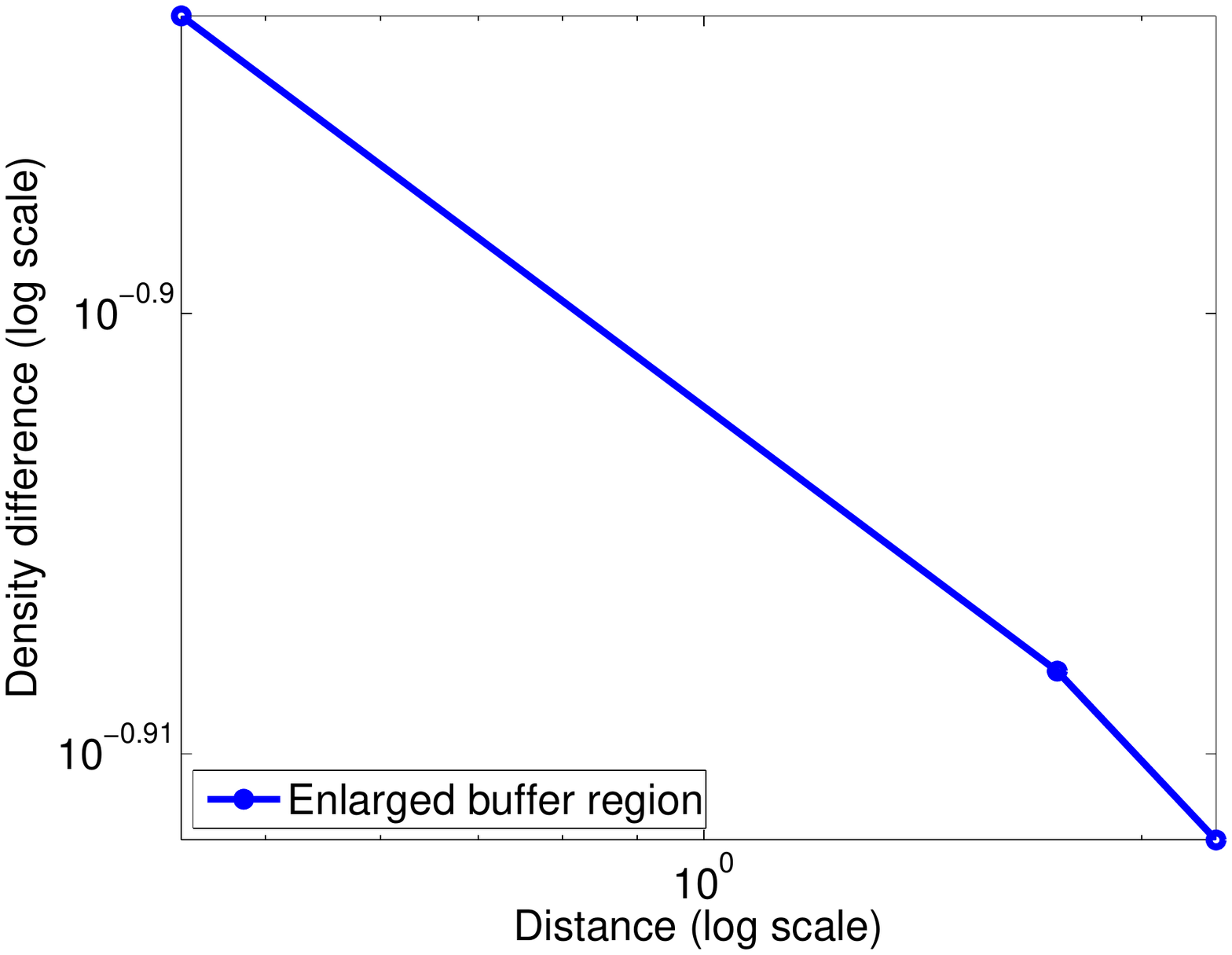}}%
\subfigure[Enlarged buffer region]{
\label {fig:insulatormetaldensitydifference2d3}
\includegraphics[width=2.0in]{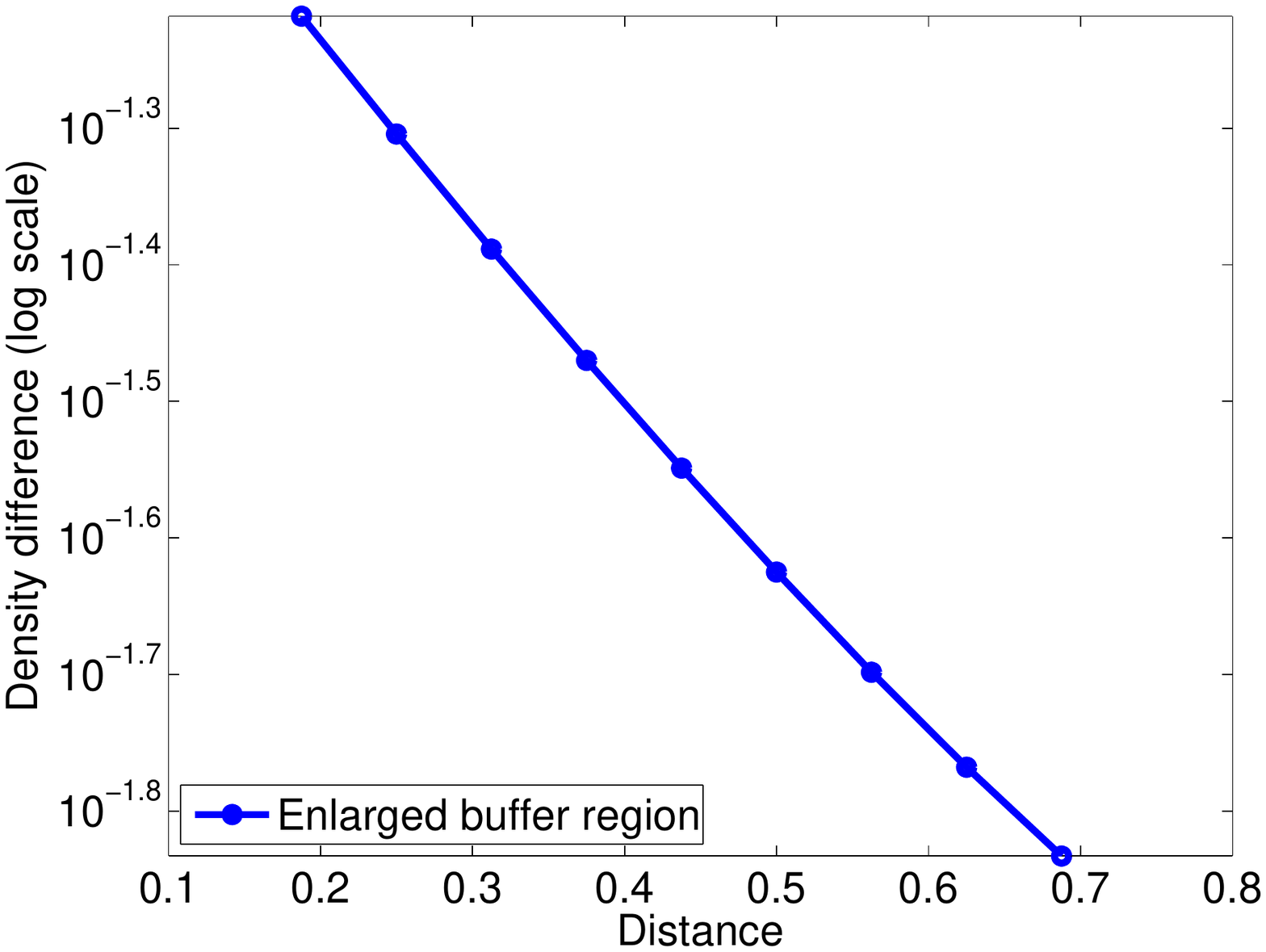}}%
\vspace{-2em}
\caption{\small Difference between electron densities of the global system and subsystems
as a function of distance in \exref{ex:insulatormetal2d}. (a) $\abs{\rho(x)-\rho_{\Lam}(x)}$ with $\Lam=[3.1,8.9]\times[3.1,8.9]$ and $\Lam=[3, 9]\times[3,9]$ as a function of $x$; (b) $\abs{\rho(5,5)-\rho_{\Lam}(5,5)}$ with $\Lam=[5,7]\times[5,7]$ and $\Lam_b=[5-x,7+x]\times[5-x,7+x]$ ($x\ge0.1$)
as a function of $x$; (c) $\abs{\rho(3,3)-\rho_{\Lam}(3,3)}$ with $\Lam=[3,9]\times[3,9]$ and $\Lam_b=[3-x,9+x]\times[3-x,9+x]$ ($x\ge0.1$) as a function of $x$. Exponential decay rate is observed in (c), while only algebraic decay rates are observed in (a) and (b) due to the validity and invalidity of \eqref{assumpc} in corresponding cases.}\label {fig:insulatormetaldensitydifference2d}
\end{figure}
\end{example}

\section{Conclusion}

In this work, we identify the crucial gap assumption for both the
global system and the subsystem for the accuracy of the
DAC method for electronic structure calculations.
Under the gap assumption, we prove that the pointwise difference
between electron densities of the global system and the subsystem
decays exponentially as a function of the distance away from the
boundary of the subsystem. This analytic conclusion is verified by
numerical examples.

From a physical point of view, our result suggests that while the
DAC method works quite well for insulating systems, one
still needs to be careful in the choice of subdomain and restrictions
to guarantee the gap assumption. Moreover, for heterogeneous systems
with large local Fermi energy variations, such as
metal-insulator-metal bilayer devices or systems  involving long range
charge transfer, application of the DAC method might
need extra care.

Finally, let us emphasize that our accuracy estimate only depends on
the size of the gap and the $L^{\infty}$ norm of the effective
potential. Hence, even though the discussion here focuses on the
DAC method, the analysis allows for general restriction
of Hamiltonian, and hence can be applied to a variety of methods in
electronic structure calculations using the domain decomposition idea.

\medskip
\noindent \textbf{Acknowledgment.}  We thank Professor Weinan E for
suggesting the problem and for his encouragement. J.L. would also like
to thank Professor Weitao Yang for helpful discussions. The work of
J.C. was supported in part by the National Science Foundation via
grant DMS-1217315. The work of J.L. was supported in part by the
Alfred P.~Sloan Foundation and the National Science Foundation under
grant no.~DMS-1312659.



\end{document}